\documentclass[12pt]{amsart}

\usepackage[sort,numbers,square]{natbib}
\usepackage{amsmath,amsthm,amsfonts,amssymb,bbm}
\usepackage{url}
\usepackage{dsfont,extarrows,enumerate}
\usepackage{hyperref}
\usepackage{fullpage}
\newcommand{\E}{\mathbf E}
\newcommand{\R}{\mathbb R}

\newcommand{\sB}{\mathcal B}
\newcommand{\sK}{\mathcal K}
\newcommand{\sI}{\mathcal I}
\newcommand{\sF}{\mathcal F}
\newcommand{\sG}{\mathcal G}

\newcommand{\sM}{\mathcal M}

\newcommand{\bM}{\mathbf M}

\newcommand{\MM}{\mathbb M}

\newcommand{\eps}{\varepsilon}

\newcommand{\USC}{\mathsf{USC}_0}
\newcommand{\contf}{\mathsf{C}_0}

\newcommand{\ind}{\mathds{1}}
\newcommand{\eqd}{\overset{d}{\sim}}
\newcommand{\Prob}[1]{\mathbf{P}\left\{#1\right\}}
\renewcommand{\P}{\mathbf{P}}
\newcommand{\salg}{\mathfrak{F}}
\newcommand{\sme}{\mathfrak{S}}
\newcommand{\smi}{\mathfrak{S}_{\mathrm{ind}}}
\newcommand{\carrier}{\mathbb E}
\newcommand{\cpc}{\phi}

\theoremstyle{plain} \newtheorem{theorem}{Theorem}[section]
\theoremstyle{plain} \newtheorem{proposition}[theorem]{Proposition}
\theoremstyle{plain} \newtheorem{lemma}[theorem]{Lemma}
\theoremstyle{plain} \newtheorem{corollary}[theorem]{Corollary}
\theoremstyle{definition} \newtheorem{definition}[theorem]{Definition}
\theoremstyle{definition} 
\theoremstyle{remark} \newtheorem{remark}[theorem]{Remark}
\theoremstyle{remark} \newtheorem{example}[theorem]{Example}

\let\phi\varphi

\usepackage{fancybox}
\newlength{\querylen}
\usepackage{tikz}
\usetikzlibrary{positioning}
\usetikzlibrary{shapes}
\usetikzlibrary{snakes}
\usetikzlibrary{arrows}

\usetikzlibrary{decorations.markings}
\tikzset{
  b arrow/.style={
    decoration={markings,mark=at position 1 with {\arrow[scale=2]{>}}},
    postaction={decorate}}
}
\tikzset{
  bb arrow/.style={
    decoration={markings,
    mark=at position 0.1 with {\arrow[scale=2]{<}}},
    decoration={markings,
    mark=at position 1 with {\arrow[scale=2]{>}}},
    postaction={decorate}}
}

\begin{document}

\title[Max-stable random sup-measures]{Max-stable random sup-measures\\ with comonotonic tail dependence}

\author[I. Molchanov]{Ilya Molchanov}
\address{University of Bern \\
 Institute of Mathematical Statistics and Actuarial Science \\
 Sidlerstrasse 5 \\
 CH-3012 Bern \\
 SWITZERLAND}
\email{ilya.molchanov@stat.unibe.ch}

\author[K. Strokorb]{Kirstin Strokorb}
\address{University of Mannheim \\
 Institute of Mathematics\\
 A5, 6 B 116\\
 D-68131 Mannheim\\
 GERMANY }
\email{strokorb@math.uni-mannheim.de}

\thanks{IM supported in part by Swiss National Science Foundation
  grant 200021-153597.}



\date{\today}

\begin{abstract}
  Several objects in the Extremes literature are 
  special instances of max-stable random sup-measures.
  This perspective opens connections to the theory of random sets
  and the theory of risk measures and makes it possible to
  extend corresponding notions and results from the literature 
  with streamlined proofs. In particular, it clarifies the role of 
  Choquet random sup-measures and their stochastic dominance property. 
  Key tools are the LePage representation of a max-stable random sup-measure
  and the dual representation of its tail dependence functional.  
  Properties such as complete randomness, 
  continuity, separability, coupling, continuous choice,
  invariance and transformations are also analysed.
\end{abstract}

\maketitle

\vspace{10mm}
\section{Introduction}

{Random sup-measures} provide a unified framework for dealing with a
number of objects that naturally appear in the Extremes literature
including temporal extremal processes \cite{obr:tor:ver90,lac:sam15},
continuous choice models \cite{res:roy91} or extremal loss in portfolios
\cite{yuen:st14}; \mbox{$\alpha$-Fr{\'e}chet} sup-measures are the
building blocks of max-stable processes \cite{stoev:taq05}. In
general, any stochastic process with upper semicontinuous paths can be
viewed as a random sup-measure 
\cite{obr:tor:ver90,nor86,ver97}. That is, the suprema
of the process over sets yield a random sup-measure, while the values of the random sup-measure at singletons yield the upper semicontinuous process. The max-stability property of the process immediately translates into the same property
of the random sup-measure.

Surprisingly, the notion of capacities and sup-measures has almost
vanished from the theoretical developments on extremal processes over
the past 20 years.
This paper aims to clarify, extend and complement a number of results  
from the unifying perspective of sup-measures and capacities with
streamlined proofs and connections to the theory of random sets and
utility functions (or risk measures).  The necessary preliminaries
concerning capacities, random closed sets, random sup-measures, Choquet and
extremal integrals are presented in Section~\ref{sec:capac-sup-meas}.

Section~\ref{sec:max-stable-random} introduces \emph{max-stable random
sup-measures} $X$ on a carrier space $\carrier$ and their 
\emph{tail dependence functional} $\ell$, which
are the central objects in this paper.  They are natural
generalisations of max-stable random vectors and their (stable) tail
dependence functions. For a given function $f$ on $\carrier$, the tail
dependence functional $\ell(f)$ characterises the distribution of the
extremal integral of $f$ with respect to $X$ and so uniquely
determines the distribution of the random sup-measure $X$. The values
$\ell(\ind_K)$ of $\ell$ on indicator functions $\ind_K$ are called
\emph{extremal coefficients} and denoted by $\theta(K)$.

Generalising \cite{res13,res11m,mo08e} we give a \emph{complete
  characterisation} of the tail dependence functional as an upper
semicontinuous homogeneous max-completely alternating functional
and of the extremal coefficient functional as an upper semicontinuous
union-completely alternating functional.
Motivated by the family of stochastic processes studied in
\cite{str:sch15} and characterised by the fact that their
distributions are in a one-to-one correspondence with extremal coefficient functionals, we identify the family of max-stable random sup-measures that have the same property.  While in \cite{str:sch15} such processes were called
Tawn--Molchanov processes (TM processes), 
here we call their sup-measure analogues \emph{Choquet random sup-measures (CRSMs)}. The key argument relies on the fact that the \emph{comonotonic additivity} property of the 
tail dependence functional $\ell$ ensures that
$\ell$ equals the Choquet integral with respect to $\theta$, and so
the distribution of the random sup-measure is uniquely determined by
$\theta$. 
This observation clarifies a number of properties of TM
processes from \cite{str:sch15} and establishes connections with 
the studies of coherent risk
measures that also appear as such Choquet
integrals, see \cite{delb12,foel:sch04}.
The following graph illustrates the one-to-one correspondence between 
extremal coefficient functionals and CRSMs, cf.\
\cite{str:sch15}.

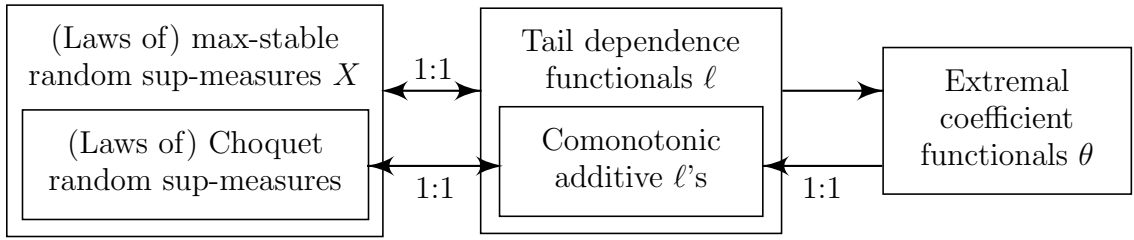
\begin{center}
\begin{tikzpicture}[>=latex']
  \tikzstyle{block} = [inner sep= 1mm, rectangle, draw=black, semithick, text centered]
  \node[block,draw] (nodeA)           
       {\begin{minipage}{4.8cm}
           \centering 
           $\phantom{a}$\\[-2mm] (Laws of) max-stable\\ random sup-measures $X$\\[2mm]
           \begin{minipage}{4.6cm}
             \centering
             \begin{tikzpicture}
               \node[block,draw] {\begin{minipage}{4.4cm}
                   \centering
                   $\phantom{a}$ \\[-2mm] (Laws of) Choquet random sup-measures \\[-2mm] $\phantom{a}$ \end{minipage}};
             \end{tikzpicture}
           \end{minipage}\\[-1mm] 
           $\phantom{a}$
       \end{minipage}};
       \node[block,draw,node distance=5.8cm] (nodeB) [right of=nodeA]           
            {\begin{minipage}{3.8cm}
                \centering 
                $\phantom{a}$\\[-2mm] Tail dependence\\ functionals $\ell$ \\[2mm] 
                \begin{minipage}{3.6cm}
                  \centering
                  \begin{tikzpicture}
                    \node[block,draw] {\begin{minipage}{3.3cm}
                        \centering
                        $\phantom{a}$ \\[-2mm] Comonotonic additive {$\ell$}'s \\[-2mm] $\phantom{a}$ \end{minipage}};
                  \end{tikzpicture}
                \end{minipage}\\[-1mm] 
                $\phantom{a}$
            \end{minipage}};
            \node[block,draw,node distance=5cm] (nodeC) [right of=nodeB] 
                 {\begin{minipage}{3.1cm}
                     \centering 
                     $\phantom{a}$\\[-2mm]
                         {Extremal}\\{coefficient}\\
                         {functionals} $\theta$\\[-2mm]
                         $\phantom{a}$
                 \end{minipage}};
                 \draw[bb arrow,thick] (nodeA.east)+(0.22,0.4) to node[auto] {1:1\phantom{a}} +(1.3,0.4);
                 \draw[b arrow,thick] (nodeB.east)+(0,0.4) to  +(1.35,0.4);
                 \draw[b arrow,thick] (nodeC.west)+(0,-0.6) to 
                 node[below] {} node[below] {1:1} +(-1.6,-0.6);
                 \draw[bb arrow,thick] (nodeB.west)+(0.1,-0.6) to 
                 node[below] {} node[below] {\phantom{a}1:1} +(-1.5,-0.6);
\end{tikzpicture} 
\end{center}

The classical \emph{LePage series} representation
\cite{lep:wood:zin81} asserts that a symmetric stable random vector
equals in distribution the sum of i.i.d.\ integrable random vectors
scaled by successive points of the unit intensity Poisson process on
the positive half-line. Its variant for max-stable processes is
derived in \cite{dehaan84}. In Section~\ref{sec:seri-repr-stable}, we
derive such a representation of a general max-stable random sup-measure 
as the maximum of i.i.d.\ copies of a random sup-measure scaled by successive
Poisson points. The difficulty lies in the absence of a norm and a reference 
sphere in the space of (locally finite) sup-measures.
Subsequently, CRSMs are characterised by the fact that the
i.i.d.\ summands become scaled indicator random sup-measures.

Section~\ref{sec:dual-representation} provides the \emph{dual representation} 
of the general and CRSM tail dependence functionals
as supremum over the Lebesgue integrals with respect a certain family
of Radon measures.  In the CRSM case, this family has an
interpretation as distributions for selections of a random closed set.
Such dual representations are related to dual representations of coherent
risk measures in mathematical finance.
For random vectors (when the carrier space $\carrier$ is finite), 
these families of measures are convex bodies that were called 
dependency sets or max-zonoids in \cite{mo08e} and \cite{str:sch15}.
Among all tail dependence functionals with fixed values on indicator
functions (that is, with fixed extremal coefficients), 
the CRSM tail dependence functional is the largest one.

Random sup-measures with independent values on disjoint sets are
called \emph{completely random} or having independent peaks
\cite{nor86,stoev:taq05}. They are now well understood including the
corresponding integration theory that relies on the concept of the
\emph{extremal integral} \cite{stoev:taq05}. The distribution of a
max-stable completely random sup-measure is completely identified by
its \emph{control measure}, similarly to the situation with
conventional $\alpha$-stable completely random measures studied in
details in \cite{sam:taq94}.  In the completely random case, the tail
dependence functional $\ell(f)$ equals the Lebesgue integral $\int f
d\mu$ with respect to the control measure $\mu$ and so it is
comonotonic additive.
Therefore, max-stable completely random sup-measures belong to the family
of CRSMs.  In Section~\ref{sec:complete-randomness} it is shown that,
conversely, a CRSM can always be realised as a max-stable completely random
sup-measure if uplifted to the (much richer) space of all closed sets.

Section~\ref{sec:max-stable-processes} addresses max-stable processes
that appear by taking the values of max-stable random sup-measures at singletons and their \emph{separability properties}. It also characterises the stochastic
continuity of a CRSM and the corresponding TM process. 
Section~\ref{sec:ordering-coupling} deals with \emph{coupling} of general max-stable random sup-measures with CRSMs. 
In particular, by means of an appropriate coupling, it is possible to recover the independence of the argmax-set of a max-stable random sup-measure from its
maximal value (and their distributions) in a streamlined proof and in
a broader setup compared to the separable \emph{continuous choice}
models on compact spaces in \cite{res:roy91}.

Finally, Section~\ref{sec:invar-transf} elaborates on further properties of both general max-stable random sup-measures and CRSMs related to
\emph{transformations} of their distributions \emph{using Bernstein functions},
\emph{rearrangement invariance} that corresponds to the law invariance
property of risk measures in finance and is related to exchangeability
properties, \emph{stationarity} and \emph{self-similarity}.
Several \emph{examples} of CRSMs are presented in Section~\ref{sec:examples-tm-sup}, in particular, related to the recent study of random sup-measures in \cite{lac:sam15}.

\vspace{10mm}
\section{Capacities, random sets and random sup-measures}
\label{sec:capac-sup-meas}

Let $\carrier$ be a locally compact Hausdorff second countable space,
that we often assume to be the line $\R$ or the Euclidean space
$\R^d$. Denote by $\sK$, $\sF$, $\sG$, and $\sB$ respectively the
families of compact, closed, open, and Borel sets in $\carrier$.

A \emph{(Choquet) capacity} $\cpc$ is a non-decreasing function $\cpc$
on the family of all subsets of $\carrier$ with values in
$[0,\infty]$, such that $\cpc(\emptyset)=0$,
$\cpc(A_n)\uparrow\cpc(A)$ if $A_n\uparrow A$ (inner regularity), and
$\cpc(K_n)\downarrow\cpc(K)$ for compact sets $K_n\downarrow K$ (upper
semicontinuity on compact sets), see \cite[Appendix~A.II]{doob84} and
\cite[Appendix~E]{mo1}. It is assumed throughout that all capacities
take finite values on compact sets. 
The capacity $\cpc$ is said to be
\emph{finite} if $\cpc(\carrier)<\infty$ and \emph{normalised} if
$\cpc(\carrier)=1$.\\

\paragraph{\textbf{\upshape Complete alternation.}}

A set-function $\cpc:\sK \rightarrow [0,\infty]$ is said to be \emph{completely alternating} on $\sK$ if the following recursively defined functionals
\begin{align*}
  \Delta_{K_1}\cpc(K)&=\cpc(K)-\cpc(K\cup K_1)\,,\\
  \Delta_{K_n}\cdots\Delta_{K_1}\cpc(K)
  &=\Delta_{K_{n-1}}\cdots\Delta_{K_1}\cpc(K) 
  -\Delta_{K_{n-1}}\cdots\Delta_{K_1}\cpc(K\cup K_n)
\end{align*}
are non-positive for all $n\geq1$ and all $K,K_1,\dots,K_n\in\sK$, see
\cite[Def.~1.1.8]{mo1}. This definition corresponds to the complete
alternation of the set-function $\cpc$ on $\sK$ considered a
semigroup with the union operation, see \cite{ber:c:r}. In particular,
the non-positivity of $\Delta_{K_1}\cpc(K)$ is equivalent to the
monotonicity of $\cpc$; together with the non-positivity of
$\Delta_{K_2}\Delta_{K_1}\cpc(K)$ they identify strongly
subadditive (or concave) set-functions. 
If a set-function $\cpc$ on $\sK$ is strongly subadditive and upper semicontinuous,  then it can be consistently extended to a Choquet capacity on the
family of all subsets of $\carrier$, see \cite[Th.~A.II.7]{doob84}.\\

\paragraph{\textbf{\upshape Random closed sets.}}

A \emph{random closed set} $\Xi$ in $\carrier$ is a measurable map
from a probability space $(\Omega,\salg,\P)$ to $\sF$ endowed with the
$\sigma$-algebra generated by the family
\begin{align*}
  \sF_K=\{F\in\sF : \; F\cap K\neq\emptyset\}\,,\quad K\in\sK.
\end{align*}
The \emph{Choquet theorem} from the theory of random sets (see
\cite[Sec.~1.2]{mo1} and \cite[Th.~6.6.19]{ber:c:r}) modified for not
necessarily finite capacities as in \cite[Th.~2.3.2]{sch:weil08}
states that $\cpc$ is a completely alternating upper semicontinuous
capacity if and only if there exists a unique locally finite measure
$\nu_\cpc$ on the family $\sF'=\sF\setminus\{\emptyset\}$ of non-empty
closed sets such that
\begin{align}
  \label{eq:choquet-th}
  \nu_\cpc(\sF_K)=\cpc(K)\,,\quad K\in\sK\,.
\end{align}
If $\cpc$ is normalised, then $\nu_\cpc$ is a probability
measure on $\sF'$. This probability measure is the distribution of a
random closed set $\Xi$ in $\carrier$ such that
\begin{align*}
  \Prob{\Xi\cap K\neq\emptyset}=\cpc(K)\,,\quad K\in\sK,
\end{align*}
and $\cpc$ is then called the \emph{capacity functional} of $\Xi$. \\

\paragraph{\textbf{\upshape Maxitive capacities.}}

A capacity $\cpc$ is called a \emph{sup-measure} if 
\begin{align*}
  \cpc(\cup_{j\in J} G_j)=\sup_{j\in J} \cpc(G_j)
\end{align*}
for any family $\{G_j, j\in J\}$ of open sets. This is the case if and
only if $\cpc$ is obtained as the extension of an upper semicontinuous
set-function on $\sK$ such that 
\begin{align}
  \label{eq:maxitive}
  \cpc(K_1\cup K_2)=\cpc(K_1)\vee \cpc(K_2),\qquad K_1,K_2\in\sK,
\end{align}
and so $\phi$ is called \emph{maxitive} on $\sK$, see
\cite{nor86}. Note that $\vee$ denotes the maximum operation, for
random vectors it denotes the coordinatewise maximum, and for
functions their pointwise maximum. Each maxitive capacity $\cpc$ is
completely alternating, see \cite[Th.~1.1.17]{mo1}, and
\begin{align}
  \label{eq:sup-der}
  \cpc(K)=\sup\{g(x):\; x\in K\}, \qquad K\in\sK,
\end{align}
for an upper semicontinuous function $g$, see \cite[Prop.~1.1.16]{mo1}
and \cite[Th.~2.5]{obr:tor:ver90}.  The right-hand side of
\eqref{eq:sup-der} is denoted by $g^\vee(K)$ and is called the
\emph{sup-integral} of $g$, while the function $g(x)=\cpc(\{x\})$,
$x\in\carrier$, is the \emph{sup-derivative} of $\cpc$, see e.g.\
\cite{obr:tor:ver90}.  A particularly important maxitive capacity is
the \emph{indicator capacity} $\cpc(K)=\ind_{F\cap K\neq\emptyset}$
for any fixed $F\in\sF$. \\

\paragraph{\textbf{\upshape Choquet and extremal integrals.}}

The \emph{Choquet integral} of a function
$f:\carrier\mapsto\R_+=[0,\infty)$ with respect to a capacity $\cpc$
is defined by
\begin{align}
  \label{eq:choquet-int}
  \int fd\cpc=\int_0^\infty \cpc(\{f\geq t \})d t \,,
\end{align}
where $\{f\geq t \}=\{x \in \carrier:\; f(x)\geq t \}$, see
\cite{den94} and \cite[Sec.~5.1]{mo1}. If $\cpc$ is a measure and $f$
is a measurable non-negative function, this integral coincides with
the Lebesgue integral. 
Furthermore,
\begin{equation}
  \label{eq:2.1}
  \int fd\cpc=\int_0^\infty \cpc(\{f>t\})dt \,,
\end{equation}
since the function $\cpc(\{f\geq t \})$ is monotone in $t$, and so has
at most a countable number of discontinuities if $\cpc(\{f\geq t \})$
is finite for all $t$, while otherwise the both sides are infinite,
see \cite[Eq.~(6)]{ger96v} and \cite[Th.~42, p.~123]{delb12}.

The \emph{extremal integral} 
\begin{align}
  \label{eq:int-gerritse}
  \int^e fd\cpc=\sup\{\cpc(K)\inf_{x\in K} f(x):\; K\in\sK\}
\end{align}
was introduced in \cite{ger96v} in view of applications
to the theory of large deviations. It is shown in
\cite[Prop.~3]{ger96v} that
\begin{align}
  \label{eq:sup-integral-strict}
  \int^e fd\cpc=\sup_{t > 0} t\, \cpc(\{f\geq t \})
  =\sup_{t > 0} t\, \cpc(\{f> t \})\,.
\end{align}
If $\cpc=g^\vee$ is the sup-integral of an upper semicontinuous
function $g$, then
\begin{align*}
  \int^e fdg^\vee =\sup_{x \in \carrier} f(x)g(x)\,,
\end{align*}
which justifies calling this integral the extremal one. In particular,
if $\cpc$ is a sup-measure and $f=\bigvee_{i=1}^n a_i\ind_{A_i}$, then
\begin{equation}
  \label{eq:2.2}
  \int^e f d\cpc =\max_{i=1,\dots,n} a_i\cpc(A_i)\,.
\end{equation}

By $\USC$ (respectively $\contf$) we denote the family of all
non-negative bounded upper semicontinuous (respectively continuous)
functions on $\carrier$ with relatively compact support $\{x \in
\carrier:\; f(x)\neq 0\}$.  Both the Choquet integral and extremal
integral are finite if the integrand belongs to $\USC$ or if the
integrand is bounded and $\cpc(\carrier)$ is finite.

\begin{lemma}
  \label{lemma:int-theta-nu}
  If $\cpc$ and $\nu_\cpc$ are related by \eqref{eq:choquet-th}, then,
  for each $f\in\USC$,
  \begin{align*}
    \int f d\cpc = \int f^\vee d\nu_\cpc
    \qquad \text{and} \qquad
    \int^e f d\cpc = \int^e f^\vee d\nu_\cpc .
  \end{align*}
\end{lemma}
\begin{proof} 
  It suffices to note that 
  \begin{align*}
    \cpc(\{f\geq t \})
    =\nu_\cpc(\{F:\; F\cap\{f\geq t \} \neq \emptyset \})
    =\nu_\cpc(\{F:\; f^\vee(F) \geq t \})
  \end{align*}
  and to apply the definitions of the Choquet and extremal integrals.
\end{proof}

The following result is known for the Choquet integral from
\cite[Th.~43, p.~124]{delb12} and \cite[Ch.~8]{den94} and for the
extremal integral it follows from the upper semicontinuity 
of~$\cpc$.

\begin{lemma}
  \label{lemma:approx}
  If $f_n(x)\downarrow f(x)$ for all $x\in \carrier$, and
  $f,f_1,f_2,\ldots\in\USC$, then
  \begin{align*}
    \int f_n d\cpc\downarrow \int f d\cpc\quad
    \text{and}\quad \int^e f_n d\cpc\downarrow \int^e f d\cpc\quad 
    \text{as}\; n\to\infty\,.
  \end{align*}
  In particular, the values of the integrals on any $f \in \USC$ can
  be approximated by their values on step-functions that approximate
  $f$ from above. 
\end{lemma}

By approximating $f\in\USC$ with step-functions, it is easy to see
that, for all $f \in \USC$ and any sup-measure $\cpc$, the integral
given by \eqref{eq:int-gerritse} coincides with the extremal integral
introduced in \cite{stoev:taq05}.\\

\paragraph{\textbf{\upshape Comonotonic additivity.}}

Both the Choquet integral and the extremal integral are 
homogeneous, e.g.\
\begin{align*}
  \int (cf)d\cpc=c\int fd\cpc\,,\qquad c\geq 0\,.
\end{align*}
While the Choquet integral is not a linear
functional of $f$, it is \emph{comonotonic additive} meaning that
\begin{align*}
  \int(f+g)d\cpc=\int fd\cpc+\int gd\cpc
\end{align*}
for two comonotonic functions $f$ and $g$. Recall that $f$ and $g$ are
\emph{comonotonic} if
\begin{align*}
  (f(x)-f(y))(g(x)-g(y))\geq 0
\end{align*}
for all $x,y\in \carrier$, see \cite[Prop.~5.1]{den94} and
\cite{schmeid86}, where it is also shown that each normalised
comonotonic additive monotone functional can be represented as a
Choquet integral. The subadditivity property of the Choquet integral
\begin{align*}
  \int (f+g)d\cpc\leq \int fd\cpc+\int gd\cpc
\end{align*}
is equivalent to the concavity property of $\cpc$, see
\cite[Th.~6.3]{den94}.  \\
 
\paragraph{\textbf{\upshape Random sup-measures.}}

A sequence $\{\cpc_n, n\geq1\}$ of sup-measures converges
\emph{sup-vaguely} to $\cpc$ if $\limsup_{n\to\infty} \cpc_n(K)\leq
\cpc(K)$ and $\liminf_{n\to\infty} \cpc_n(G)\geq \cpc(G)$ for
all $K\in\sK$ and $G\in\sG$, see e.g.\ \cite[Def.~2.6]{obr:tor:ver90}
and \cite{ver97}.
The sup-vague topology generates the Borel $\sigma$-algebra on the family
of sup-measures and so makes it possible to define a \emph{random
  sup-measure} $X$. 
Its distribution is determined by the joint
distributions of random vectors $X(K_1),\dots,X(K_m)$ for all finite
collections of compact sets $K_1,\dots,K_m$. These distributions form
the system of \emph{finite-dimensional distributions} of $X$.  A
random sup-measure $X$ is said to be \emph{integrable} if $X(K)$ is
integrable for all $K\in\sK$.

By approximating $f\in\USC$ from above using step-functions it is easy
to see that the Choquet and extremal integrals of $f$ with respect to
a random sup-measure are random variables.

\begin{lemma}
  \label{lemma:distr-sup-measure}
  The distribution of a random sup-measure $X$ is uniquely determined by
  the distributions of $\int^e f dX$ for all $f\in\USC$.
\end{lemma}
\begin{proof}
  If $f=\bigvee u_i\ind_{K_i}$, then $\int^e f dX=\bigvee u_iX(K_i)$
  by \eqref{eq:2.2}. Thus, it is possible to obtain 
  the joint distribution of
  $X(K_1),\dots,X(K_m)$, i.e.\ the finite dimensional distribution of
  $X$, from the distribution of $\int^e fdX$ for
  varying coefficients $u_1,\dots,u_m\in\R_+$.
\end{proof}

A random sup-measure is said to be \emph{completely random} if its
values on disjoint sets are jointly independent, see
\cite{stoev:taq05}; it is said to have independent peaks in
\cite{nor86}. 

\vspace{10mm}
\section{Max-stable random sup-measures\\ and their tail dependence functionals}
\label{sec:max-stable-random}

\paragraph{\textbf{\upshape Max-stable random vectors.}}

A random variable has a unit Fr{\'e}chet distribution if its
cumulative distribution function is $\exp\{-at^{-1}\}$, $t>0$,
where $a>0$ is called the scale parameter.  A random vector
$\xi=(\xi_1,\dots,\xi_d)$ is called \emph{semi-simple max-stable} if,
for all $u \in \R^d_+=[0,\infty)^d$, the max-linear combination
$\bigvee_{j=1}^d u_j \xi_j$ is a unit Fr{\'e}chet variable with scale
parameter denoted by $\ell(u)$, see \cite{dehaan78}.
The function $\ell: \R^d_+ \mapsto \R_+$ is called \emph{(stable) tail
  dependence function} and has the following properties, see
\cite{beir:goeg:seg:04,mo08e,res13}.
\begin{enumerate}[(i)]
\item $\ell$ is \emph{homogeneous}, i.e.\ $\ell(c u)=c \ell(u)$ for
  all $u \in \R^d_+$ and all $c >0$.
\item $\ell$ is \emph{subadditive}, i.e.\ $\ell(u+v) \leq \ell(u) +
  \ell(v)$ for all $u,v \in \R^d_+$.
\item $\ell$ is \emph{max-completely alternating}, i.e.\ the
  successive differences
  \begin{align*}
    \Delta^{\vee}_{u_1}\ell(u)&=\ell(u)-\ell(u \vee u_1)\,,\\
    \Delta^{\vee}_{u_n}\cdots\Delta^{\vee}_{u_1}\ell(u)
    &=\Delta^{\vee}_{u_{n-1}}\cdots\Delta^{\vee}_{u_1}\ell(u) 
    -\Delta^{\vee}_{u_{n-1}}\cdots\Delta^{\vee}_{u_1}\ell(u\vee u_n)
  \end{align*}
  are all non-positive for all $n\geq1$ and all
  $u,u_1,\dots,u_n\in\R^d_+$.
\end{enumerate}
Since $\ell$ is a sublinear (homogeneous and subadditive) function, it
defines a norm on $\R^d_+$ called a $D$-norm in
\cite{aul:fal:hof13,falk06}.  In fact, the homogeneity and
max-complete alternation suffice to characterise the tail
dependence function as can be seen from a slight modification of
\cite[Th.~6]{res13} and \cite[Th.~4]{res11m}, see also
\cite[Th.~7]{mo08e}.

\begin{theorem}
  \label{thr:ell-sub}
  A function $\ell:\R^d_+\mapsto\R_+$ is a tail dependence function of
  a semi-simple max-stable random vector $\xi$ in $\R^d$ if and only if
  $\ell$ is homogeneous and max-completely alternating.
\end{theorem}

\paragraph{\textbf{\upshape Max-stable random sup-measures.}}
\label{sec:semi-simple-max}

A random sup-measure $X$ is called \emph{semi-simple max-stable} (in
the sequel we say that $X$ is a \emph{max-stable random sup-measure}) 
if its finite-dimensional distributions $X(K_1),\dots,X(K_m)$ 
are semi-simple max-stable random vectors for all $K_1,\dots,K_m\in\sK$, $m\geq1$.

\begin{lemma}
  A random sup-measure $X$ is semi-simple max-stable if and only if
  $\int^e f dX$ is a unit Fr\'echet random variable for all $f \in
  \USC$.
\end{lemma}
\begin{proof}
  \textsl{Necessity}.  The statement is true for functions taking a
  finite number of values and then by approximation for $f \in \USC$
  using Lemma~\ref{lemma:approx}.

  \textsl{Sufficiency}.  For any $u_1,\dots,u_m\geq0$ and
  $K_1,\dots,K_m\in\sK$, the random variable $\bigvee u_iX(K_i)$ is
  equal to $\int^e f dX$ for $f=\bigvee u_i\ind_{K_i} \in \USC$, which
  is unit Fr{\'e}chet distributed. Hence $(X(K_1),\dots,X(K_m))$ is a
  semi-simple max-stable random vector.
\end{proof}

\paragraph{\textbf{\upshape Tail dependence functional.}}

In the sequel the scale parameter of the unit Fr\'echet variable
$\int^e f dX$ will be denoted by $\ell(f)$.
The function $\ell:\USC\mapsto\R_+$ is called the \emph{tail
  dependence functional} of $X$. By
Lemma~\ref{lemma:distr-sup-measure}, the tail dependence functional
uniquely determines the law of $X$.

\begin{theorem}
  \label{thr:ell-sup-measure}
  A functional $\ell:\USC\mapsto\R_+$ is the tail dependence functional of
  a (necessarily unique) max-stable random sup-measure if and only if $\ell$ is
  homogeneous, completely alternating on $\USC$ equipped with the
  maximum operation, and upper semicontinuous on $\USC$ meaning that
  $\ell(f_n)\downarrow \ell(f)$ for $f_n \downarrow f$.
\end{theorem}
\begin{proof}
  \textsl{Necessity}. The homogeneity property is trivial. The values
  of $\ell$ on $f_1,\dots,f_m\in\USC$ and their partial maxima are the
  extremal coefficients of the semi-simple max-stable random vector
  $(\int^e f_1dX,\dots,\int^e f_mdX)$ and their complete alternation
  property follows from Theorem~\ref{thr:ell-sub}. If $f_n\downarrow
  f$, then $\int^e f_n dX\downarrow \int^e fdX$ a.s.\ by
  Lemma~\ref{lemma:approx}, so that $\ell(f_n)\downarrow \ell(f)$, see
  also \cite[Lemma~2.1~(iv)]{stoev:taq05}.

  \textsl{Sufficiency}. Let $X(K_1),\dots,X(K_m)$ for $K_1,\dots,K_m
  \in \sK$ be the semi-simple random vector with the tail dependence
  function $\ell_{K_1,\dots,K_m}(u)=\ell(\bigvee_{i=1}^m
  u_i\ind_{K_i})$, $u \in \R_+^m$.  The function
  $\ell_{K_1,\dots,K_m}$ is indeed a tail dependence function, since it
  inherits max-completely alternation and homogeneity from
  $\ell$. This system of finite-dimensional distributions is
  consistent, since
  \begin{align*}
    \ell\bigg(\bigvee_{i=1}^{m+1} u_i\ind_{K_i}\bigg)\downarrow 
    \ell\bigg(\bigvee_{i=1}^{m} u_i\ind_{K_i}\bigg)\qquad \text{as }
    u_{m+1}\downarrow 0.
  \end{align*}
  Thus, there exists a unique max-stable random sup-measure $X$ with the
  specified finite-dimensional distributions.  By
  Lemma~\ref{lemma:approx} and the upper semi-continuity of $\ell$,
  the tail dependence functional of $X$ coincides with $\ell$.
\end{proof}

\begin{proposition}
  The tail dependence functional $\ell$ is subadditive,
  i.e.\ $\ell(f+g)\leq \ell(f)+\ell(g)$ for all $f,g\in\USC$.
\end{proposition}
\begin{proof}
  By approximation from below, it suffices to derive the
  result for step-functions $f$ and $g$ that, without loss of
  generality, can be taken as $f=\sum a_i\ind_{K_i}$ and $g=\sum
  b_i\ind_{K_i}$ for disjoint $K_1,\dots,K_m \in \sK$.  Then $\ell(f)$
  equals the tail dependence function of the random vector
  $(X(K_1),\dots,X(K_m))$ in direction $(a_1,\dots,a_m)$ and similar
  interpretations hold for $\ell(g)$ and $\ell(f+g)$. It suffices to
  refer to the subadditivity property of the tail dependence function.
\end{proof}

\paragraph{\textbf{\upshape Extremal coefficient functional.}}

Let $X$ be a max-stable random sup-measure with tail dependence functional
$\ell$.  The set-function $\theta(K)=\ell(\ind_K)$, $K\in\sK$, will be
termed the \emph{extremal coefficient functional} of $X$. It is necessarily 
a capacity on $\sK$ as the following lemma shows. Note that
$\theta(K)$ is the scale parameter of the unit Fr\'echet law of
$X(K)$. 

\begin{lemma}
  \label{lemma:theta-ca-usc}
  A functional $\theta:\sK\mapsto\R_+$ is the extremal coefficient
  functional of a stable sup-measure if and only if it is completely
  alternating, upper semicontinuous and satisfies
  $\theta(\emptyset)=0$. 
\end{lemma}
\begin{proof}
  \textsl{Necessity.}  The (union-)complete alternation property of
  $\theta$ follows from the max-complete alternation property of the
  tail dependence functional $\ell$, cf.\
  Theorem~\ref{thr:ell-sup-measure} for functions
  $f_i(x)=\ind_{K_i}(x)$, $i=1,\dots,m$, taking into account that
  $\ind_{K_i}\vee\ind_{K_j}=\ind_{K_i\cup K_j}$. The upper
  semicontinuity of $\theta$ follows from the upper semicontinuity of
  $\ell$ noticing that $\ind_{K_n}\downarrow\ind_K$ as $K_n\downarrow
  K$. Since $\ell$ is homogeneous, $\theta(\emptyset)=0$.
  
  \textsl{Sufficiency.} Setting $\ell(f)=\int f d\theta$ the Choquet integral
  with respect to $\theta$, we see that the functional $\ell$ satisfies the 
  properties of a tail dependence functional, cf.\ 
  Theorem~\ref{thr:ell-sup-measure}.
\end{proof}

\paragraph{\textbf{\upshape Choquet random sup-measures (CRSMs).}}
In general, the information contained in the extremal coefficient functional 
$\theta$ is not sufficient to
recover the tail dependence function $\ell$ and so the distribution of
the corresponding max-stable random sup-measure. Now we single out particular
max-stable random sup-measures, whose distributions are completely characterised
by the extremal coefficient functional.

\begin{definition}
  A stable sup-measure $X$ is said to be a \emph{Choquet random sup-measure} 
  (CRSM) if its tail dependence functional $\ell$ is comonotonic additive.
\end{definition}

\begin{theorem}
  \label{thr:ell} 
  A stable sup-measure $X$ is a CRSM if and only if
  its tail dependence functional $\ell$ is given by the Choquet
  integral
  \begin{align}
    \label{eq:choquet-infinite}
    \ell(f)=\int fd\theta
  \end{align}
  with respect to its extremal coefficient functional $\theta$.  The
  functional $\theta$ uniquely determines the distribution of $X$.
\end{theorem}
\begin{proof}
  Sufficiency follows from the comonotonic additivity of the Choquet
  integral. For the proof of necessity, consider a sequence
  $\{K_n,n\geq1\}$ of compact sets that grows to $\carrier$. For each
  $n\geq1$, the tail dependence functional $\ell$ is comonotonic
  additive on functions $f$ supported by $K_n$ if and only if it can
  be represented as the Choquet integral with respect to a capacity
  $\theta_n$ (see \cite{schmeid86}), i.e.\ $\ell(f)=\int fd\theta_n$
  for $f$ supported by $K_n$, where $\theta_n(K)=\ell(\ind_K)$,
  $K\in\sK$, $K\subset K_n$. Thus, $\theta_n$ is the extremal
  coefficient functional of $X(K)$, $K\subset K_n$. Noticing that
  $\theta_n(K)=\theta_m(K)$ for $m>n$ and $K\subset K_n$,
  \eqref{eq:choquet-infinite} holds for $\theta(K)=\theta_n(K)$ with
  $K\subset K_n$.
\end{proof}

\begin{remark}
  CRSMs appear as weak limits for the scaled maxima of indicator
  random sup-measures. This also relates to their series representation
  derived in the following section. 
\end{remark}

\vspace{10mm}
\section{Series representations}
\label{sec:seri-repr-stable}

\paragraph{\textbf{\upshape Series representation of max-stable random sup-measures.}}
A useful tool for the study of stable random elements is their series
representation in terms of the sum (or maximum) of i.i.d.\ random
elements scaled by the successive points of the unit intensity Poisson
process, see \cite{dehaan84} and \cite{lep:wood:zin81} for the
max and sum-stable cases, and \cite{dav:mol:zuy08} for general
semigroups.  The following result provides a series decomposition for
max-stable random sup-measures.

Denote by $\lambda$ the Lebesgue measure on $\R_+$ and let
$\{\Gamma_i,i\geq1\}$ be the sequence of successive points of the unit
intensity Poisson process on $\R_+$.  Denote by $\sme$ the family of
all sup-measures on $\carrier$, and by $\smi$ the family of scaled
indicator sup-measures $c\ind_{F\cap K\neq\emptyset}$ for $c>0$ and
$F\in\sF$. Their non-trivial subsets will be denoted by
$\sme'=\sme\setminus\{0\}$ and $\smi'=\smi\setminus\{0\}$,
respectively.

\begin{theorem}
  \label{thr:lepage-sup-measures}
  A random sup-measure $X$ is max-stable if and only if it can be decomposed as a max-series
  \begin{align}
    \label{eq:lepage-sm}
    X \eqd \bigvee_{i\geq1} \Gamma_i^{-1} Y_i,
  \end{align}
  where $\{Y_i,i\geq1\}$ is a sequence of i.i.d.\ copies of an integrable random
  sup-measure $Y$ and independent of the sequence $\{\Gamma_i,i\geq 1\}$.
  The tail dependence functional $X$ is then given by
  \begin{align}
    \label{eq:tail-x-lepage}
    \ell(f)=\E \int^e fdY,\qquad f\in\USC. 
  \end{align}
  The random sup-measure $X$ is a.s.\ non-trivial if and only if $Y$ is a.s.\
  non-trivial.
\end{theorem}
\begin{proof}
  \textsl{Sufficiency}. If $X$ is given by the right-hand side of
  \eqref{eq:lepage-sm}, then 
  \begin{align*}
    \int^e fdX=\bigvee_{i\geq1} \Gamma_i^{-1} \int^e fdY_i
  \end{align*}
  is a unit Fr\'echet random variable with the scale parameter
  $\ell(f)$ given by \eqref{eq:tail-x-lepage}, which is finite if
  $f\in\USC$ and $Y$ has integrable values on compact sets.

  \textsl{Necessity.} It suffices to consider the case of an 
  a.s.\ non-trivial random sup-measure $X$.
  Note that a max-stable random sup-measure is necessarily
  max-infinitely divisible.  By \cite[Th.~5.1]{nor86} and noticing that
  the support of the distribution of $X(K)$ is the whole $\R_+$, 
  the sup-measure $X$ can be represented as
  \begin{align}
    \label{eq:norberg}
    X\eqd \bigvee_{i\geq 1} \eta_i,
  \end{align}
  where $\{\eta_i,i\geq1\}$ form a Poisson process with the unique
  intensity measure $\Lambda$ on $\sme'$, that is called the L\'evy
  measure.

  At this point it is useful to view the space $\sme$ as a convex cone
  which is the abelian semigroup with the semigroup operation being
  maximum and the scaling given by scaling the values of sup-measures,
  see \cite{ber:c:r} and \cite{dav:mol:zuy08}. A separating family of
  semicharacters on $(\sme,\vee)$ is given by
  $\chi_{K,a}(\cpc)=\ind_{\cpc(K)\leq a}$, $\cpc\in\sme$, $K\in\sK$,
  and $a>0$. This means that two different sup-measures yield
  different values for a semicharacter from this family. 
  It is easy to see that condition \textbf{(C)} of
  \cite{dav:mol:zuy08} is satisfied, while \eqref{eq:norberg} means
  that the L\'evy measure of $X$ is supported by $\sme$ in the
  terminology of \cite{dav:mol:zuy08}.  By
  \cite[Th.~6.1]{dav:mol:zuy08}, $\Lambda$ is $1$-homogeneous with
  respect to scaling, i.e.\
  \begin{align*}
    \Lambda(\{c \cpc:\; \cpc\in B\})=c^{-1}\Lambda(B),\qquad c>0,
  \end{align*}
  for all Borel $B\subset\sme'$.  Let $\{K_i,i\geq1\}$ be the closures
  of relatively compact sets that form a countable base for the
  topology of $\carrier$.  Then
  \begin{align*}
    -\log\Prob{X(K_i)\leq a_i,\, i=1,\dots,m} =\Lambda(\{\varphi \,:\,
    \max_{i=1,\dots,m}\varphi(K_i) > a_i \}),
  \end{align*} 
  for $a_1,\dots,a_m>0$ and $m\geq1$. By repeating an argument from
  the proof of \cite[Th.~1]{dehaan84}, there exist $b_i>0$, $i\geq1$,
  such that $\Lambda$ is supported by $\sme_b=\{\varphi \in \sme'
  \,:\, r(\varphi)<\infty \}$, where $r(\phi)=\sup_{i\geq 1} b_i
  X(K_i)$.
  Denote $S=\{\varphi \in \sme_b \,:\, r(\varphi)=1\}$ 
  and define the map $T:\sme_b \mapsto (0,\infty)
  \times S$, by letting $T(\varphi)=(r(\varphi),\varphi/r(\varphi))$,
  whose inverse is simply $T^{-1}(r,\varphi)=r\varphi$. By the
  homogeneity property of $\Lambda$ on $\sme_b$ and the homogeneity of
  $r$, it is easily seen that the push-forward of $\Lambda$ under $T$
  is the product measure $u^{-2}du\otimes \pi(d\varphi)$ for a finite
  measure $\pi$ on $S$ given by
  \begin{align*}
    \pi(B)=\Lambda( \{\varphi \,:\, r(\varphi)>1, \,
    \varphi/r(\varphi) \in B \}).
  \end{align*}
  If the $b_i$ are scaled by the same constant, $\pi$ can be adjusted
  to become a probability measure on $\sme_b$.  Conversely, $\Lambda$
  is fully determined by $\pi$ through
  \begin{align*}
    \Lambda(\{\varphi \,:\, \max_{i=1,\dots,m}\varphi(K_i) > a_i \})
    = \int_{S} \bigvee_{i=1}^{m} \frac{\varphi(K_i)}{a_i} \pi(d\varphi).
  \end{align*}
  Finally, \eqref{eq:lepage-sm} follows by letting $Y_i$ be i.i.d.\
  with distribution $\pi$.
\end{proof}

\begin{remark}
  \label{rem:t-vs-gamma}
  The intensity measure $\Lambda$ of the Poisson process
  $\{\eta_i,i\geq1\}$ from \eqref{eq:norberg} is a homogeneous measure
  on $\sme'$. Sometimes, $\Lambda$ is decomposed as the push-forward
  of the product of the measure with density $t^{-2}$ on $(0,\infty)$
  and a not necessarily finite measure $\nu$ on $\sme'$. Then, instead
  of \eqref{eq:lepage-sm}, one obtains the representation
  $\bigvee_{i\geq1} t_i^{-1}Y_i$, where $\{(t_i,Y_i),i\geq1\}$ is the
  Poisson process on $\R_+\times\sme'$ with intensity measure
  $\lambda\otimes\nu$. The special feature of \eqref{eq:lepage-sm} is the fact
  that such a Poisson process can be viewed as the Poisson process on
  the positive half-line marked by i.i.d.\ copies of a random
  sup-measure.
\end{remark}

\begin{remark}
  \label{rem:non-unique-Y}
  The distribution of $Y$ in Theorem~\ref{thr:lepage-sup-measures},
  i.e.\ the probability measure $\pi$ on $\sme'$ that was constructed
  in the proof, is said to be the \emph{spectral measure} of $X$.  The
  spectral measure is not unique, e.g.\ it is possible to replace $Y$
  with $\zeta Y$, where $\zeta$ is any non-negative random variable
  independent of $Y$ with the unit expectation. Two random
  sup-measures, $Y$ and $Y'$, yield the same max-stable random sup-measure if
  $\E\int^e fdY=\E\int^e fdY'$ for all $f\in\USC$.
 
  In case of a countable carrier space
  $\carrier=\{x_i,i\geq1\}$, it means that the sequences
  $\{Y(\{x_i\}),i\geq1\}$ and $\{Y'(\{x_i\}),i\geq1\}$ are zonoid
  equivalent, see \cite{mol:sch:stuc13}. The uniqueness of $\pi$ (and
  $Y$) is achieved if the values of $Y$ are normalised, e.g.\ by
  assuming that $Y\in S$ as introduced in the proof of
  Theorem~\ref{thr:lepage-sup-measures}.  

  If $X(\carrier)$ is a.s.\
  finite, the uniqueness can be achieved by requiring that
  $Y(\carrier)=c$ for a constant $c>0$.  In this case, the proof of
  Theorem~\ref{thr:lepage-sup-measures} simplifies
  using $\varphi(\carrier)$ instead of $r(\varphi)$.\\
\end{remark}

\paragraph{\textbf{\upshape Series representation of CRSMs.}}
The following result characterises CRSMs in terms of their series
representations. Recall that $\sF'=\sF\setminus\{\emptyset\}$.

\begin{theorem}
  \label{thr:tm-sup-measures}
  A random sup-measure $X$ is a CRSM with the extremal
  coefficient functional $\theta$ if and only if
  \begin{align}
    \label{eq:rep-tm-sup-pp}
    X(K)\eqd \bigvee_{i\geq1} t_i^{-1}\ind_{F_i\cap
      K\neq\emptyset},
    \qquad K\in\sK\,,
  \end{align}
  where $\{(t_i,F_i),i\geq1\}$ is the Poisson process on
  $\R_+\times \sF'$ with intensity $\lambda\otimes\nu$ for a locally
  finite measure $\nu$ on $\sF'$ such that
  \begin{align}
    \label{eq:nu-theta-1}
    \nu(\sF_K)=\theta(K),\qquad K\in\sK. 
  \end{align}
\end{theorem}
\begin{proof}
  \textsl{Sufficiency}. 
  A random sup-measure given by \eqref{eq:rep-tm-sup-pp} is
  necessarily semi-simple max-stable. The local finiteness of $\nu$
  implies that at most a finite number of pairs $(t_i,F_i)$ satisfy
  $F_i\cap K\neq\emptyset$ and $t_i\leq s$ for any $K\in\sK$ and
  $s\geq0$, so that $X(K)$ is almost surely finite. For any
  $f\in\USC$,
  \begin{align*}
    \int^e fdX
    \eqd \bigvee_{i\geq1} t_i^{-1} f^\vee(F_i)
  \end{align*}
  is the series representation of the unit Fr\'echet random
  variable. In order to find its scale parameter we calculate the void
  probability of the Poisson process $\{(t_i,F_i)\}$ as follows
  \begin{align*}
    \Prob{\int^e fdX< s}&=
    \exp\{-(\lambda\otimes\nu)(\{(t,F):\; f^\vee(F)t^{-1}\geq s\})\}\\
    &=\exp\bigg\{-\int_0^\infty \nu(\{F:\; f^\vee(F)\geq ts\})dt\bigg\}\\
    &=\exp\bigg\{-s^{-1}\int_0^\infty \nu(\{F:\; f^\vee(F)\geq t\})dt\bigg\}\\
    &=\exp\bigg\{-s^{-1}\int f^\vee d\nu\bigg\}.
  \end{align*}
  By Lemma~\ref{lemma:int-theta-nu},
  $\ell(f)=\int fd\theta$ for $\theta$ given by
  \eqref{eq:nu-theta-1}, and so $\ell$ is comonotonic.
  
  \textsl{Necessity}. By Lemma~\ref{lemma:theta-ca-usc} and the
  Choquet theorem, there exists a unique measure $\nu$ on $\sF'$ that
  satisfies \eqref{eq:nu-theta-1}. The random sup-measure constructed
  by \eqref{eq:rep-tm-sup-pp} has the tail dependence functional $\int
  fd\theta$, which equals $\ell(f)$ by Theorem~\ref{thr:ell}.
\end{proof}

\begin{corollary}
  \label{cor:finite-lepage}
  A CRSM $X$ such that $X(\carrier)$ is almost surely
  positive and finite, can be represented as
  \begin{align}
    \label{eq:lepage-tm-finite}
    X(K)\eqd \theta(\carrier)\bigvee_{i\geq1}
    \Gamma_i^{-1}\ind_{\Xi_i\cap K\neq\emptyset},
  \end{align}
  where $\{\Xi_i,i\geq1\}$ is a sequence of i.i.d.\ a.s.\ non-empty
  random closed sets in $\carrier$ with the capacity functional
  $\Prob{\Xi_1\cap K\neq\emptyset}=\theta(K)/\theta(\carrier)$ 
  and independent of the sequence $\{\Gamma_i,i\geq 1\}$.
\end{corollary}
\begin{proof}
  Since the measure $\nu$ on $\sF'$ related to $\theta$ by
  \eqref{eq:nu-theta-1} is finite and non-vanishing, the Poisson process
  $\{(t_i,F_i)\}$ with the intensity $\lambda\otimes\nu$ can be viewed
  as the unit intensity Poisson process $\{\Gamma_i,i\geq1\}$ on
  $\R_+$ scaled by $\theta(\carrier)^{-1}$ and independently marked by
  a sequence of random elements in $\sF'$ that are distributed
  according to the normalised $\nu$.
\end{proof}

The following result characterises CRSMs as those having the spectral
measure supported by the family $\smi'$ of non-trivial scaled indicator
sup-measures.

\begin{theorem}
  \label{thr:lepage-tm}
  A non-trivial random sup-measure $X$ is a CRSM if and only if
  \eqref{eq:lepage-sm} holds with 
  $\{Y_i,i\geq1\}$ being i.i.d.\ copies of an integrable random 
  sup-measure $Y$ with distribution supported by $\smi'$ 
  and independent of the sequence $\{\Gamma_i,i\geq 1\}$.
  The extremal coefficient functional of $X$ is
  \begin{align}
    \label{eq:theta-Y}
    \theta(K)=\E Y(K),\qquad K\in\sK. 
  \end{align}
\end{theorem}
\begin{proof}
  \textsl{Sufficiency} is easy to see noticing that if
  $Y(K)=\tau\ind_{\Xi\cap K\neq\emptyset}$, then $\ell(f)=\E[\tau
  f^\vee(\Xi)]$ is comonotonic additive. By \eqref{eq:tail-x-lepage},
  \begin{align*}
    \theta(K)=\E \int^e \ind_K dY=\E Y(K).
  \end{align*}
  
  \textsl{Necessity.} 
  First, $X$ admits the representation given by \eqref{eq:lepage-sm}.
  Since $X(K_0)$ is a.s.\ finite, Corollary~\ref{cor:finite-lepage}
  applies. Therefore, $X(K)$, $K\subset K_0$, admits the
  representation as the max-series built from scaled indicator
  random sup-measures. Since the L\'evy measure of $X$, i.e.\ the intensity of
  the Poisson process that appears in \eqref{eq:norberg} is unique,
  the corresponding spectral measure $\pi$ is supported by
  $\smi'$. Thus, $Y(K)$, $K\subset K_0$, almost surely belongs to the
  family $\smi'$. The conclusion follows from the fact that $K_0$ is
  arbitrary.
\end{proof}

\begin{remark}
  \label{rem:lepage-tm-non-integrable}
  The random sup-measure $Y$ in Theorem~\ref{thr:lepage-tm} can be
  represented as $Y(K)=\tau\ind_{\Xi\cap K\neq\emptyset}$.  If
  $Y(\carrier)=\tau$ is integrable, then $\theta$ is finite and the
  LePage series \eqref{eq:lepage-tm-finite} yields a version of
  $X$. Thus, the most interesting case of Theorem~\ref{thr:lepage-tm}
  corresponds to non-integrable $\tau$, where the dependency between
  $\tau$ and $\Xi$ ensures that $Y(K)$ is integrable for all
  $K\in\sK$.  For example, if $\carrier=\R_+$ and $\Xi=[\tau,\infty)$,
  then $\E Y(K)=\E[\tau\ind_{\tau\leq \sup K}]<\infty$ for $K\in\sK$,
  no matter if $\tau$ is integrable or not.
\end{remark}

\begin{example}
  \label{eq:two-integrals}
  Consider a sup-measure $\cpc(K)=\sup\{g(x):\; x\in K\}$ for an upper
  semicontinuous function $g:\carrier\mapsto[0,1]$ and let
  $Y(K)=\ind_{\Xi\cap K\neq\emptyset}$ with random closed set $\Xi$
  that has the capacity functional $\cpc$, that is $\Xi=\{x:\;
  g(x)\geq U\}$ for the uniform random variable $U$ in $[0,1]$. Then
  \eqref{eq:lepage-sm} with $Y_i$ being i.i.d.\ copies of
  $Y$ yields the CRSM $X$ with the extremal coefficient functional
  $\cpc(K)$. If $Y_i$ are chosen to be deterministic and equal $\cpc$,
  then \eqref{eq:lepage-sm} yields the max-stable random sup-measure
  $\widetilde{X}(K)=\zeta\cpc(K)$, where $\zeta$ is the unit Fr\'echet
  random variable with scale parameter one. Thus, $X$ and $\widetilde{X}$
  share the same extremal coefficient functional, while $\widetilde{X}$
  has the tail dependence functional $\int^e fd\cpc$, which is
  not comonotonic additive and so it is not a CRSM, and the CRSM $X$ has
  the tail dependence functional $\int fd\cpc$. Their extremal
  coefficients coincide, since the Choquet and extremal integrals
  return the same value on indicator functions. 
\end{example}

\begin{corollary} 
  Let $Y$ be an integrable random sup-measure.  Then $Y\in\smi$ a.s.\
  if and only if
  \begin{align}
    \label{eq:ext-choquet}
    \E \int^e fdY=\E \int fdY,\qquad f\in\USC.
  \end{align}
\end{corollary}
\begin{proof}
  \textsl{Necessity.}  Since the left-hand side of
  \eqref{eq:ext-choquet} is the tail dependence function $\ell(f)$ of a CRSM
  constructed by \eqref{eq:lepage-sm}, it is comonotonic
  additive. The statement follows from \eqref{eq:theta-Y} and
  \eqref{eq:choquet-infinite}, so that
  \begin{align*}
    \ell(f)=\int fd\theta=\int_0^\infty \E Y(\{f\geq t \}) dt
    =\E \int fdY. 
  \end{align*}

  \textsl{Sufficiency.}  If (\ref{eq:ext-choquet}) holds, the
  left-hand side of (\ref{eq:ext-choquet}) is the tail dependence
  functional of a random sup-measure $X$ constructed by
  (\ref{eq:lepage-sm}). Since the right-hand side of
  (\ref{eq:ext-choquet}) is comonotonic additive, $X$ is a CRSM. It
  follows from Theorem~\ref{thr:lepage-tm} that $Y\in \smi$ a.s.
\end{proof}

\vspace{10mm}
\section{Dual representations}
\label{sec:dual-representation}

The following result provides a dual representation for tail
dependence functionals of max-stable random sup-measures. Denote by $\MM$ the
family of Radon measures on the Borel $\sigma$-algebra
$\sB$ in $\carrier$.

\begin{theorem}
  \label{thr:duality}
  Let $X$ be a max-stable random sup-measure. Then
  \begin{align}
    \label{eq:ell-dual}
    \ell(f)=\sup_{\mu\in\bM} \int fd\mu\,,\qquad f\in\USC,
  \end{align}
  for a convex family 
  \begin{align}
    \label{eq:ell-dual-measures}
    \bM=\{\mu\in\MM:\; \int fd\mu\leq \ell(f),\; f\in\USC\}.  
  \end{align}
\end{theorem}
\begin{proof}
  The tail dependence functional restricted to the family $\contf$ of
  continuous functions on $\carrier$ with compact support is a
  capacity in the sense of \cite[Def.~4.1]{fug71}. By
  \cite[Th.~5.3]{fug71}, \eqref{eq:ell-dual} holds for all 
  $f \in \contf$ with $\bM$ replaced by
  \begin{align*}
    \bM_c=\{\mu\in\MM:\; \int fd\mu\leq \ell(f),\; f\in\contf\}.  
  \end{align*}
  It follows from \cite[Th.~3.13]{al:bor06} and Urysohn's lemma that, 
  for all $f \in \USC$, there exists a
  sequence of functions $\{f_n,n\geq1\}$ from $\contf$ approximating $f$
  from above. Then the upper semicontinuity and Fatou's
  lemma yield that $\bM_c=\bM$. Hence, \eqref{eq:ell-dual} holds for
  all $f \in \contf$.

  In \cite[Def.~4.2]{fug71}, the functional on $\contf$ is extended to
  $\USC$ by approximation from above.  In view of the
  existence of a sequence of continuous functions approximating
  $f\in\USC$ from above, and the
  upper semicontinuity of $\ell$, we deduce that this extension of
  $\ell$ from $\contf$ to $\USC$ coincides with the originally defined
  $\ell$. By \cite[Th.~5.5]{fug71}, $\ell(f)$ is given by
  \eqref{eq:ell-dual}.
\end{proof}

\begin{remark}
  The functional \eqref{eq:ell-dual} constructed for an arbitrary
  convex family $\bM$ may fail to satisfy the complete alternation
  property, and so is not necessarily the tail dependence functional
  of a max-stable random sup-measure.
\end{remark}

\begin{proposition}
  \label{prop:tm-sets-dual}
  The functional \eqref{eq:ell-dual} is the tail dependence functional
  of a CRSM $X$ with extremal coefficient functional $\theta$ 
  if and only if $\bM=\bM_\theta$, where 
  \begin{align}\label{eq:theta-dual-measures}
    \bM_\theta=\{\mu\in\MM:\; \mu(K)\leq \theta(K),\; K\in \sK\}.  
  \end{align}
\end{proposition}
\begin{proof}
  \textsl{Necessity.} By letting $f=\ind_K$ in
  \eqref{eq:ell-dual-measures}, it is easily seen that
  $\bM\subset\bM_\theta$. If $X$ is a CRSM, then its tail
  dependence functional has the dual representation
  \eqref{eq:ell-dual} with the family $\bM$ given by
  \eqref{eq:ell-dual-measures}. By Lemma~\ref{lemma:approx} and
  Fatou's lemma, $\bM$ is the family of all $\mu\in\MM$ such that
  $\int fd\mu \leq \ell(f)$ for all step-functions $f=\sum
  a_i\ind_{K_i}$ with $a_1,\dots,a_n>0$ and $K_1\supset
  K_2\supset\cdots\supset K_n$. The comonotonic additivity of $\ell$
  yields that $\int fd\mu\leq \ell(f)$ for such functions $f$ if and
  only if $\mu(K)\leq \ell(\ind_K)=\theta(K)$ for all $K\in\sK$, i.e.\
  $\mu\in\bM_\theta$ whenever $\mu\in\bM$.

  \textsl{Sufficiency.} If $\bM=\bM_\theta$, then
  \cite[Prop.~2.3]{graf80} yields that $\ell(f)= \int fd\theta$, which
  is the tail dependence functional of the CRSM with extremal
  coefficient functional $\theta$, cf.\ Theorem~\ref{thr:ell}.
\end{proof}

\begin{corollary}
  \label{cor:domination}
  Among all laws of max-stable random sup-measures sharing the same extremal coefficient 
  functional $\theta$, the (necessarily unique) CRSM law has the largest tail dependence functional.
\end{corollary}
\begin{proof}   
  The assertion follows from Theorem~\ref{thr:duality} and
  Proposition~\ref{prop:tm-sets-dual}, since $\bM_\theta$ from
  \eqref{eq:theta-dual-measures} includes the family $\bM$ given by
  \eqref{eq:ell-dual-measures} if $\ell(\ind_K)=\theta(K)$ for $K \in
  \sK$.
\end{proof}

\begin{remark}
  If $\theta(\carrier)=1$, then $\bM_\theta$ can be further restricted
  to consist of probability distributions of all selections of the
  random closed set $\Xi$ with the capacity functional $\theta$, that
  is random elements $\xi$ in $\carrier$ such that $\xi$ and $\Xi$ can
  be realised on the same probability space to ensure that $\xi\in\Xi$
  a.s.
\end{remark}

\begin{remark}
  \label{rk:cash-inv}
  The value $X(\carrier)$ is a.s.\ finite if and only if the total
  mass of all measures from $\bM$ in \eqref{eq:ell-dual} is uniformly
  bounded. However, even in this case, $\ell(f+a)$
  is not necessarily equal to $\ell(f)+a\ell(1)$ for $a\in\R_+$, since 
  the measures $\mu$ in \eqref{eq:ell-dual} may have varying total
  masses. Max-stable random sup-measures satisfying
  $\ell(f+a)=\ell(f)+a\ell(1)$ for all $a\in\R_+$ form a family
  sandwiched between the CRSM and general max-stable random sup-measures.
  If $\ell(1)=1$, then the functional $\ell(-f)$ has the properties of
  a coherent risk measure, see \cite{delb12,foel:sch04}. In
  particular, the subadditivity property shows that diversification
  reduces risks, and $\ell(-(f+a))=\ell(-f)-a$ is called the
  cash-invariance property. This property makes it possible to extend
  $\ell$ onto the family of all bounded measurable functions. 
\end{remark}

\vspace{10mm}
\section{Complete randomness}
\label{sec:complete-randomness}

Recall that a random sup-measure is said to be \emph{completely
  random} if it assumes jointly independent values on disjoint
sets. Hence, the tail dependence functional $\ell$ of a max-stable completely random sup-measure is finitely additive on linear combinations of indicator functions of disjoint sets, and, by approximation, is
finitely additive on $\USC$. The
upper semicontinuity property yields that $\ell(f)=\int fd\mu$ for a Radon
measure $\mu$ (called \emph{control measure}) 
that necessarily coincides with the extremal coefficient
functional $\theta$. Conversely, if $\theta$ is finitely additive, then it
corresponds to a max-stable completely random sup-measure. This yields the
following result.

\begin{proposition}
  \label{prop:cr-iff-theta-add} 
  Let $X$ be a max-stable random sup-measure 
  with extremal coefficient functional
  $\theta$.  Then the following are equivalent:
  \begin{enumerate}[(i)]
  \item $X$ is completely random.
  \item $\theta$ is finitely additive.
  \item $\theta$ is a Radon measure $\mu$.
  \item The tail dependence functional of $X$ admits the
    representation~\eqref{eq:ell-dual} with $\bM$ being a singleton
    $\bM=\{\mu\}$.
  \end{enumerate}
\end{proposition}

Each max-stable completely random sup-measure $X$ is a CRSM,
and each CRSM becomes completely random if uplifted to the space
$\sF'$ of non-empty closed sets.

\begin{proposition}
  \label{prob:completely-random}
  A max-stable random sup-measure $X$ on $\carrier$ is a CRSM if and only if
  $X(K)=Z(\sF_K)$, $K\in\sK$, for a max-stable completely random 
  sup-measure $Z$ on $\sF'$.
\end{proposition}
\begin{proof}
  It follows from \eqref{eq:rep-tm-sup-pp} that $X(K)$ is obtained as
  $Z(\sF_K)$ for 
  \begin{align*}
    Z(\sM)=\sup\{t_i^{-1}:\; F_i\in\sM\} 
  \end{align*}
  for each measurable $\sM\subset\sF'$. Since $\{(t_i,F_i)\}$ is a
  Poisson process, the random sup-measure $Z$ is completely random.

  In the other direction, the equality $X(K)=Z(\sF_K)$ for all
  $K\in\sK$ yields that 
  \begin{align*}
    \int^e f d X = \int^e f^\vee d Z,
  \end{align*}
  by Lemma~\ref{lemma:int-theta-nu}, so that the tail dependence
  functional of $X$ is given by $\ell(f) = \int f^\vee d \nu$,
  where $\nu$ is the control measure of $Z$. 
  Since $(f+g)^\vee(F)=f^\vee(F)+g^\vee(F)$ for
  comonotonic functions $f$ and $g$, the functional $\ell$ is
  comonotonic. 
\end{proof}

\begin{remark}
  Proposition~\ref{prob:completely-random} together with
  Lemma~\ref{lemma:int-theta-nu} can be used to replace the integral
  $\int^e fdX$ with $\int^e f^\vee dZ$, where the latter integral is
  taken for a completely random sup-measure and so can be extended for
  all integrands $f$, such that $f^\vee$ is integrable with respect to
  the control measure of $Z$, see \cite{stoev:taq05}.
\end{remark}

\begin{theorem}
  \label{thr:cr-transform}
  For each CRSM $X$, there is a set-valued function
  $F:[0,1]\mapsto\sF$ such that $X(K)=Z(F^-(K))$ for a completely
  random sup-measure $Z$ on $[0,1]$ and $F^-(K)=\{u\in[0,1]:\;
  F(u)\cap K\neq\emptyset\}$. 
\end{theorem}
\begin{proof}
  Corollary~\ref{cor:limits} and the upper semicontinuity of $\theta$
  yield that $X$ is separable in probability as a process indexed by $\sK$, 
  that is, it satisfies
  Condition S, see \cite{stoev:taq05} and
  \cite{sam:taq94}. Applying \cite[Th.~3]{dehaan84}, we obtain that 
  \begin{displaymath}
    X(K)\eqd \bigvee_{i\geq1} \Gamma_i^{-1} f_K(U_i),
    \qquad K\in\sK\,,
  \end{displaymath}
  for a Poisson process $\{(\Gamma_i,U_i)\}$ on $\R_+\times[0,1]$. By
  Theorem~\ref{thr:lepage-tm}, $X$ is a CRSM if and only if
  $f_K(U_i)=\tau_i\ind_{\Xi_i\cap K\neq\emptyset}=\tau_i\ind_{U_i\in F^-(K)}$, where
  $\Xi_i=F(U_i)$ for some set-valued function $F$. 
  Thus, $X(K)=Z(F^-(K))$, $K \in \sK$ for the  completely random 
  sup-measure $Z(A)=\bigvee_{i\geq1} \Gamma_i^{-1} \tau_i \ind_{U_i \in A}$.
\end{proof}

\vspace{8mm}
\section{Max-stable processes, separability and continuity}
\label{sec:max-stable-processes}

\paragraph{\textbf{\upshape Max-stable processes.}}
The sup-derivative $\xi(x)=X(\{x\})$, $x \in \carrier$, of a max-stable random sup-measure is a max-stable process on $\carrier$ with upper semicontinuous paths. Conversely, sup-integrals of max-stable process with unit Fr{\'e}chet marginals and upper semicontinuous paths are max-stable random sup-measures.
Sup-derivatives of CRSMs are called \emph{TM processes}
in \cite{str:sch15}. If $\carrier$ is finite, then the values of a
CRSM on its points build a TM random vector, see
Example~\ref{ex:tm-vectors}.

It should be noted that information on a random sup-measure $X$ can be lost when passing to  its sup-derivative $\xi$. For instance, $\xi(x)$ may almost surely vanish for all $x\in\carrier$ while $X$ is positive almost surely on all compact balls (with positive radius). This is the case e.g.\ if $X$ is completely random with a non-atomic control measure that is positive on such balls. Because of this, the max-stable random sup-measures provide a more general setting  compared to max-stable processes as studied by their finite-dimensional distributions.\\

\paragraph{\textbf{\upshape Separability.}} 
A random sup-measure $X$ (and the corresponding
functionals $\ell$ and $\theta$) is called \emph{separable} if the
distribution of $X$ is uniquely determined by the finite-dimensional
distributions of its sup-derivative $\xi(x)$ for $x$ from a countable
set $D\subset\carrier$, that is
\begin{align}
  \label{eq:x-separant}
  X(G)=\sup_{x\in D\cap G} X(\{x\})\qquad \text{a.s.}, \quad G \in \sG.
\end{align}
By expressing the both sides of \eqref{eq:x-separant} using the LePage
series \eqref{eq:rep-tm-sup-pp}, it is easily seen that a CRSM $X$ is
separable if and only if
\begin{align}
  \label{eq:separable}
  \nu(\sF_G)=\nu(\sF_{D\cap G}),\qquad G\in \sG,
\end{align}
where $\nu$ is the measure on $\sF'$ associated with the extremal
coefficient functional $\theta$ of $X$ by \eqref{eq:nu-theta-1}.

Let $\sI$ be the family of finite subsets of $\carrier$. A completely
alternating functional $\theta$ on $\sI$ with $\theta(\emptyset)=0$
can be extended to the capacity on $\sK$ by letting
\begin{align}
  \label{eq:t-tilde-1}
  \tilde\theta(G)&=\sup\{\theta(I):\; I\subset G,\; I\in\sI\},\qquad
  G\in\sG,\\
  \label{eq:t-tilde-2}
  \tilde\theta(K)&=\inf\{\tilde\theta(G):\; K\subset G,\;G\in\sG\},\qquad K\in\sK.
\end{align}

\begin{proposition}
  \label{prop:separab-1}
  Let $\theta$ be a completely alternating functional on $\sI$ 
  with $\theta(\emptyset)=0$. Then
  $\tilde\theta$ given by \eqref{eq:t-tilde-1} and
  \eqref{eq:t-tilde-2} is the smallest extremal coefficient functional
  that dominates $\theta$. The CRSM with the extremal coefficient
  functional $\tilde\theta$ is separable. Finally, $\theta$ is the
  restriction on $\sI$ of a separable extremal coefficient functional
  if and only if $\theta$ and $\tilde\theta$ coincide on $\sI$. 
\end{proposition}
\begin{proof}
  Let $\theta'$ be another extremal coefficient functional that
  dominates $\theta$ on $\sI$. Then $\theta'$ dominates $\tilde\theta$
  on $\sG$ and so on $\sK$. Let $\nu$ be the measure on $\sF'$
  determined by $\tilde\theta$. 
  Let $B$ be any set from a countable base of the topology on $\carrier$,
  and let $\{I_n\}$ be an increasing sequence of finite sets such that
  $\theta(I_n)\uparrow\tilde\theta(B)$. Since $I_n\uparrow D_B=\cup I_n$, 
  we have $\theta(I_n)\uparrow \nu(\sF_{D_B})$.
  Therefore, $\tilde\theta(B)=\nu(\sF_{D_B})$. Finally
  \eqref{eq:separable} holds for $D$ being the union of $D_B$ over all
  $B$ from the countable base of the topology and $G$ also belonging
  to the base of topology. Its validity can be then easily extended
  for all open $G$.
\end{proof}

\paragraph{\textbf{\upshape Continuity.}} 
The series representations of max-stable
random sup-measures yield the corresponding series representations for
max-stable processes.  Since these series for TM processes involve
indicator functions, it is easy to see that TM processes are never
path continuous unless they are a.s.\ constant.

\begin{proposition}
  \label{prop:bound-contin}
  If $X$ is a CRSM with the extremal coefficient functional
  $\theta$, then, for all $K_1,K_2\in\sK$,
  \begin{align*}
    \Prob{X(K_1) - X(K_2)\leq  \eps} 
    \geq \exp\left\{- \frac{1}{\eps} 
      \left(\theta(K_1\cup K_2)-\theta(K_2)\right) \right\}. 
  \end{align*}
  In particular,
  \begin{align}
    \label{eq:bound-difference}
    \Prob{|X(K_1) - X(K_2)|> \eps} 
    \leq \frac{1}{\eps}\left({2 \theta(K_1 \cup K_2) - \theta(K_1) - \theta(K_2)}\right).
  \end{align}
\end{proposition}
\begin{proof}
  Since $X$ is a CRSM, 
  \begin{align*}
    \Prob{X(K_1)\leq p\eps, \, X(K_2)\leq q \eps}=\exp\left\{-
      \frac{\theta_{12}-\theta_2}{p\eps}
        -\frac{\theta_{12}-\theta_1}{q\eps}
        -\frac{\theta_1+\theta_2-\theta_{12}}{(p \wedge q)\eps}\right\},
  \end{align*}
  where $\theta_i=\theta(K_i)$ and $\theta_{ij}=\theta(K_i\cup
  K_j)$ and $p\wedge q=\min(p,q)$. Hence, for any $n\geq1$,
  \begin{align*}
    &\P\{X(K_1) - X(K_2)\leq \eps\}\\
    &\geq \sum_{k=1}^n \Prob{X(K_1)\leq k\eps, \, X(K_2)\leq k \eps} - \Prob{X(K_1)\leq k\eps, \, X(K_2)\leq (k-1) \eps}\\
    &= \sum_{k=1}^n \exp\left\{-\frac{1}{\eps}
      \left(\frac{\theta_{12}-\theta_2}{k}+\frac{\theta_2}{k}\right)\right\}
    -\exp\left\{-\frac{1}{\eps} \left(\frac{\theta_{12}-\theta_2}{k}
        +\frac{\theta_2}{k-1}\right)\right\} \\
    &= \sum_{k=1}^n \exp\left\{-\frac{\theta_{12}-\theta_2}{\eps
        k}\right\} \left[\exp\left\{-\frac{\theta_2}{\eps k}\right\}
      -\exp\left\{-\frac{\theta_2}{\eps(k-1)}\right\} \right]\\
    &\geq \exp\left\{-\frac{\theta_{12}-\theta_2}{\eps} \right\}
    \sum_{k=1}^n \left[ \exp\left\{-\frac{\theta_2}{\eps k}\right\}
      -\exp\left\{-\frac{\theta_2}{\eps(k-1)}\right\} \right],
  \end{align*}
  where the last telescoping sum equals
  $\exp\left\{-\theta_2/(\eps n)\right\}$ and converges to $1$.
\end{proof}

\begin{corollary}
  \label{cor:limits}
  If $X$ is a CRSM, then $X(K_n)$ converges in probability to
  $X(K)$ for $K\in\sK$ and a sequence $K_n\in\sK$, $n\geq1$, if and
  only if $\theta(K_n)\to\theta(K)$ and $\theta(K_n\cup
  K)\to\theta(K)$.
\end{corollary}
\begin{proof}
  Sufficiency follows from \eqref{eq:bound-difference}. For the
  necessity, note that the convergence in probability yields the
  convergence in distribution, and so $\theta(K_n)\to\theta(K)$. Since
  $X(K_n\cup K)\to X(K)$ in probability, $\theta(K_n\cup K)\to
  \theta(K)$.
\end{proof}

\begin{corollary}
  A CRSM is continuous in probability in the Hausdorff metric if and
  only if its extremal coefficient functional is continuous in the
  Hausdorff metric. Then $X$ is almost surely continuous at each
  $K\in\sK$ that coincides with the closure of its interior.
\end{corollary}
\begin{proof}
  If $K$ is regular closed and $K_n$ converges to $K$ in the Hausdorff
  metric, then
  \begin{math}
    K^{-\eps_n}\subset K_n\subset K^{\eps_n}
  \end{math}  
  for a sequence $\eps_n\downarrow 0$, where $K^r=\{x:\; B_r(x)\cap
  K\neq\emptyset\}$ and $K^{-r}=\{x:\; B_r(x)\subset K\}$ for the
  closed ball $B_r(x)$ of radius $r$ centred at $x$. Note that
  both $X(K^{-\eps_n})$ and $X(K^{\eps_n})$ are monotone sequences
  that converge in probability to $X(K)$ and so almost surely as well.  
\end{proof}

\begin{corollary}
  The TM process $\xi(x)=X(\{x\})$, $x\in \carrier$, of a CRSM $X$ is
  stochastically continuous if and only if $\theta(\{x\})$,
  $x\in\carrier$, is continuous.
\end{corollary}

\vspace{10mm}
\section{Coupling and continuous choice}
\label{sec:ordering-coupling}

\paragraph{\textbf{\upshape Ordered coupling.}} Two random sup-measures $X$ and $X'$ are said to admit the \emph{ordered coupling} (notation $X\preceq X'$) if
they can be realised on the same probability space so that with
probability one $X(K)\leq X'(K)$ for all $K\in\sK$. For this, one
needs that
\begin{align*}
  \Prob{X(K_1)\geq t_1,\dots,X(K_m)\geq t_m}
  \geq \Prob{X'(K_1)\geq t_1,\dots,X'(K_m)\geq t_m}
\end{align*}
for all $m\geq1$ and $K_1,\dots,K_m\in\sK$, see e.g.\
\cite{kam:kre:obr77}.

\begin{theorem}
  \label{prop:coupling}
  Let $X$ be a max-stable random sup-measure $X$, such that $X(\carrier)$ is
  a.s.\ finite. Then there exist unique CRSMs $X_*$ and $X^*$, such
  that $X_*\preceq X\preceq X^*$, and for any other CRSMs $X'$ and
  $X''$ such that $X' \preceq X \preceq X''$, we have also $X'\preceq
  X_*$ and $X^*\preceq X''$.
\end{theorem}
\begin{proof}
  The max-stable random sup-measure $X$ admits the LePage representation
  \eqref{eq:lepage-sm}, where $\{Y_i\}$ are i.i.d.\ copies of a
  sup-measure $Y$. Then $X\preceq X'$ for another max-stable random sup-measure $X'$
  if and only if $X'$ has the
  LePage representation with i.i.d.\ random sup-measures $\{Y'_i\}$
  distributed as $Y'$ such that $Y\preceq Y'$.  Since $X(\carrier)$ is
  finite,
  \begin{align*}
    \ell(\ind_\carrier)=\E \int^e \ind_\carrier dY=\E Y(\carrier) <\infty.
  \end{align*}
  Thus, $Y(\carrier)$ is a.s.\ finite. The minimal CRSM $X^*$
  dominating $X$ arises if the corresponding $Y^*$ is chosen to be the
  smallest random sup-measure $Y^*$ with realisations from $\smi$ that
  dominates $Y$, that is $Y^*(K)=Y(\carrier)\ind_{Y(K)>0}$ for
  $K\in\sK$.  Furthermore, $Y_*(K)=Y(\carrier)$ for $K\subset \{x:\;
  Y(\{x\})=Y(\carrier)\}$ and $Y_*(K)=0$ otherwise is the largest
  indicator random sup-measure that is dominated by $Y$. Finally, construct
  the CRSMs $X^*$ and $X_*$ by \eqref{eq:lepage-sm} using i.i.d.\
  copies of $Y^*$ and $Y_*$, respectively.
\end{proof}

\paragraph{\textbf{\upshape Continuous choice.}} 
Upper semicontinuous max-stable processes $\xi(x)$ defined for $x$
from a compact space $\carrier$, have been used to model
\emph{continuous choice} in \cite{res:roy91}. In particular, the
random set
\begin{align*}
  M=\{x\in\carrier:\; \xi(x)=\xi^\vee(\carrier)\}
  =\{x\in\carrier:\; X(\{x\})=X(\carrier)\}
\end{align*}
is the set of optimal choices, where $X=\xi^\vee$ is the sup-integral
of $\xi$. The upper semicontinuity assumption on $X$ and $\xi$ yields
that $M$ is indeed a random closed set \cite[Th.~1.2.27(ii)]{mo1}. 
In the following we relax the compactness assumption on $\carrier$ by
only imposing that the max-stable random sup-measure is finite, so that
$\xi^\vee(\carrier)<\infty$ and $M \neq \emptyset$ almost surely.  The
following theorem immediately recovers and extends a number of results
from \cite{res:roy91}, namely Theorem~4.1 and~4.2 and Corollary~4.1
therein, for not necessarily separable max-stable random sup-measures
on not necessarily compact spaces.  We do not need to assume that $X$
is separable, since $M$ is defined for each $\omega$ from the
probability space.

\begin{theorem}
  \label{thr:r-roy}
  If $X$ is a finite max-stable random sup-measure 
  with the LePage representation
  \eqref{eq:lepage-sm}, then the set of optimal choices $M$ is
  distributed as 
  \begin{align*}M_Y=\{x:\; Y(\{x\})=Y(\carrier)\}\end{align*} 
  and is independent of $X(\carrier)$. The set $M$ is a singleton if and only
  if $M_Y$ is a.s.\ a singleton; in particular, if $X$ is a CRSM, this
  is possible if and only if $X$ is completely random.
\end{theorem}
\begin{proof}
  If $X$ is a finite CRSM, then \eqref{eq:lepage-tm-finite}
  yields that
  \begin{align*}
    \xi(x)\eqd \theta(\carrier)\bigvee_{i\geq1} \Gamma_i^{-1} \ind_{x\in\Xi_i}
  \end{align*}
  for an i.i.d.\ sequence $\{\Xi_i\}$ of a.s.\ non-empty random closed
  sets. Then $M=\Xi_1$, and so its distribution can be identified as
  $\Prob{M\cap K\neq\emptyset}=\theta(K)/\theta(\carrier)$,
  $K\in\sK$. Furthermore, the random sets $M=\Xi_1$ and
  $\xi^\vee(\carrier)=\theta(\carrier)\Gamma_1^{-1}$ are independent.
  
  The largest CRSM $X_*$ dominated by $X$ and constructed in
  Theorem~\ref{prop:coupling} shares with $X$ the same values for
  the maximum and the corresponding arg-max set $M$. It suffices to
  note that $X_*$ admits the LePage representation
  \eqref{eq:lepage-tm-finite} and so its value on $\carrier$ and the
  arg-max set are independent, while $M$ has the same distribution as
  $\Xi_1$.

  Finally, $M$ is a singleton if and only if $M_Y$ is a singleton. In
  the CRSM case, this is equivalent to $\theta$ being a Radon measure,
  see Proposition~\ref{prop:cr-iff-theta-add}.
\end{proof}

\vspace{10mm}
\section{Invariance and transformations}
\label{sec:invar-transf}

\paragraph{\textbf{\upshape Bernstein functions.}}

A non-negative function $g$ on $[0,\infty)$ is called a
\emph{Bernstein function} if and only if it is continuous on
$[0,\infty)$ and its derivatives $g^{(n)}$ on $(0,\infty)$
exist and satisfy $(-1)^{(n+1)}g^{(n)} \geq 0$ for all $n\geq 1$. Each
Bernstein function such that $g(0)=0$ can be represented as
\begin{align}
  \label{eq:bernstein}
  g(t) = bt + \int_0^\infty (1-e^{-s t})~\varrho(ds)
\end{align}
for $b \geq 0$ and a Radon measure $\varrho$ on $(0,\infty)$ with
$\int_0^\infty (s \wedge 1)\varrho(ds)<\infty$, where $s\wedge
1=\min(s,1)$, see \cite{sch:son:von10}. Such functions can be also
viewed as the continuous non-negative negative definite functions on
the semigroup being $[0,\infty)$ with the arithmetic addition, see
\cite{ber:c:r}. A wealth of material on Bernstein functions including
many examples can be found in \cite{sch:son:von10}.

\begin{proposition}
  \label{prop:bernstein}
  Let $\theta$ be the extremal coefficient functional of a max-stable
  random sup-measure. If $g$ be a Bernstein function such that $g(0)=0$, then
  the composition $(g\circ\theta)(K)=g(\theta(K))$, $K\in\sK$, is an
  extremal coefficient functional.
\end{proposition}
\begin{proof}
  It is easily seen that $g \circ \theta (\emptyset)=0$ and $g \circ
  \theta$ is upper semicontinuous by the continuity of $g$. The
  functional $g \circ \theta$ is also completely alternating, since
  the complete alternation and negative definiteness are equivalent on
  the idempotent (and in particular 2-divisible) semigroup $(\sK,
  \cup, \emptyset)$ \cite[Cor.~4.6.8 and p.~120]{ber:c:r} and
  Bernstein functions preserve this property \cite[Prop.~3.2.9 and
  p.~114]{ber:c:r}. By Lemma~\ref{lemma:theta-ca-usc} the functional
  $g\circ \theta$ is an extremal coefficient functional.
\end{proof}

\begin{example}
  \label{ex:power} 
  The Bernstein functions $g_\alpha(t)=t^\alpha$, $\alpha \in (0,1)$
  can be represented in the form (\ref{eq:bernstein}) with
  $\varrho(ds)=\frac{\alpha}{\Gamma(1-\alpha)} s^{-(\alpha +
    1)} ds$, see \cite[p.~78]{ber:c:r}. Hence, if $\theta$ is an
  extremal coefficient functional, so is $\theta^\alpha$ for $\alpha
  \in (0,1)$.
\end{example}

\paragraph{\textbf{\upshape Rearrangement invariance.}}

Assume that $X$ is a max-stable random sup-measure and let $\mu$ be a Radon
measure on the Borel $\sigma$-algebra in $\carrier$.  The tail
dependence functional $\ell(f)$ (and the max-stable random sup-measure $X$) is
said to be \emph{rearrangement invariant} (or symmetric) with respect
to $\mu$ if $\ell(f_1)=\ell(f_2)$ if $\mu(\{f_1\geq t
\})=\mu(\{f_2\geq t \})$ for all $t>0$. 

If $X$ is a CRSM, then $\ell(f)$ is rearrangement invariant if and
only if $\theta(K)$ is \emph{symmetric}, meaning that
$\theta(K_1)=\theta(K_2)$ whenever $\mu(K_1)=\mu(K_2)$. For general stable sup-measures the symmetry of
$\theta$ does not imply the rearrangement invariance of $\ell$, see
\cite[Ex.~4]{was:kad92}.

\begin{example}
  Let $\mu$ be the counting measure on $\carrier=\{1,\dots,d\}$, so
  that a CRSM $X$ is determined by $X(\{i\})=\xi_i$, $i=1,\dots,d$. The
  rearrangement invariance of $X$ means that $\theta(K)$ depends only
  on $\mu(K)$, i.e. the cardinality of $K$. The normalised functional
  $\theta(K)$ is the capacity functional of a random set
  $\Xi\subset\{1,\dots,d\}$. By the rearrangement invariance,
  conditionally on $\mu(\Xi)=k$, the random set $\Xi$ can equally
  likely be any $k$-tuple of points from $\{1,\dots,d\}$. Thus,
  \begin{align*}
    \theta(K)
    &=c \left(\sum_{k=1}^{d-m} 
      \frac{\binom{d}{k}-\binom{d-m}{k}}{\binom{d}{k}}p_k
      + \sum_{k=d-m+1}^d p_k\right)
      =c \left(1-p_0-\sum_{k=1}^{d-m} 
      \frac{\binom{d-m}{k}}{\binom{d}{k}}p_k\right), 
  \end{align*}
  where $m$ is the cardinality of $K$, $p_0,\dots,p_d$ is a
  probability distribution on $\{0,\dots,d\}$, and
  $c=\theta(\carrier)$. 
\end{example}

\begin{example}
  Let $\carrier$ be a countable set with the discrete topology and the
  counting measure $\mu$. Assume that $\theta(\carrier)$ is finite, so
  that the normalised $\theta$ defines a random closed set $\Xi$. The
  capacity functional of $\Xi$ is rearrangement invariant if and only
  if the sequence $\{\ind_{i\in\Xi},i\geq1\}$ is exchangeable. By the
  de Finetti theorem such a sequence is conditionally i.i.d.\ Thus,
  given a random variable $\zeta\in[0,1]$, $\Xi$ consists of all
  points in $\carrier$, independently chosen with probability
  $\zeta$ and
  \begin{align*}
    \theta(K)=c(1-\E[(1-\zeta)^{\mu(K)}])
  \end{align*}
  for $c>0$ yields all rearrangement invariant finite extremal
  coefficient functionals on a countable space. 
The random set $\Xi$
  with the capacity functional given by the normalised $\theta$ is the
  support of the Cox (doubly stochastic Poisson) process whose
  intensity measure is given by $(-\log(1-\zeta))\mu$. If $\zeta=1$,
  then $\Xi=\carrier$.
\end{example}

Any rearrangement invariant extremal coefficient functional can be
represented as $\theta(K)=g(\mu(K))$ for a monotone function
$g:\R_+\mapsto\R_+$ such that $g(0)=0$ and $\theta$ is completely
alternating. This is the case, for instance, if $g$ is a Bernstein
function.

\begin{example}
  \label{example:avar}
  The rearrangement invariant tail dependence functional $\ell$ of a
  CRSM can be extended to $L^\infty$ and then, applied to $-f$,
  becomes a coherent risk measure of $f$, see \cite{delb12}.  
One of
  the most important coherent risk measures is the average value at
  risk that appears if $\theta(K)=g_\alpha(\mu(K))$ for
  $g_\alpha(t)=\frac{1}{\alpha}(t\wedge \alpha)$ with a fixed
  $\alpha\in(0,1]$ and a probability measure $\mu$, see
  \cite[Ex.~4.65]{foel:sch04}. However, $g_\alpha(\mu(K))$ is not
  alternating of order $3$ and so is not completely alternating and
  consequently is not an extremal coefficient functional. 
  Indeed,
  assume that $\mu$ is non-atomic and fix disjoint sets
  $K_1,K_2,K_3,K_4$ with equal measures $p/4$ for some $p \in (0,1]$
  such that $\alpha \in [(3/4)p,p)$.  Then
  \begin{align*}
    &\Delta_{K_1}\Delta_{K_2}\Delta_{K_3}
    (g_\alpha \circ \mu)(K_{4})
    = \frac{1}{\alpha} \left(\sum_{k=0}^{3} (-1)^k \binom{3}{k}
      \left(\frac{(k+1)p}{4} \wedge \alpha \right)\right)\\
    &= \frac{1}{\alpha} \left(\frac{p}{4} \wedge \alpha
      - 3 \frac{2p}{4} \wedge \alpha
      + 3 \frac{3p}{4} \wedge \alpha
      - \frac{4p}{4} \wedge \alpha \right) =\frac{p}{\alpha}-1 >0.
  \end{align*}
  In particular, this example shows that the convex set $\bM$ of
  probability measures having density with respect to $\mu$ bounded by
  a constant $c>1$ does not yield a tail dependence functional
  by \eqref{eq:ell-dual}. 
\end{example}

Under the assumption that the reference measure $\mu$ is a probability
measure and $\theta(\carrier)=1$, each rearrangement invariant
extremal coefficient functional can be expressed as
\begin{align*}
  \theta(K)=\int_{(0,1]}s^{-1}(\mu(K)\wedge s)\kappa(ds)+
  \ind_{K\neq \emptyset}\kappa(\{0\})
\end{align*}
for a probability measure $\kappa$ on $[0,1]$, see
\cite[Th.~4.87]{foel:sch04} and \cite{kus01}. 
Example~\ref{example:avar} shows that $\kappa$ concentrated at a
single point does not yield a valid extremal coefficient functional.

\begin{example}
  If $g(t)=t^\alpha$ for $\alpha\in(0,1)$, then
  $\kappa(dt)=\alpha(1-\alpha)t^{\alpha-1} dt$. The corresponding
  functional $\theta(K)=\mu(K)^\alpha$ is an extremal coefficient
  functional and $\ell(-f)$ is the proportional hazard risk measure.
\end{example}

\paragraph{\textbf{\upshape Stationarity and self-similarity.}}
\label{sec:stat-self-simil}

A random sup-measure on $\carrier=\R^d$ 
is called \emph{stationary} if $X(\cdot +x) \eqd X(\cdot)$ for all $x
\in \carrier$. It is called \emph{self-similar} with exponent $H$ if
$X(c \,\cdot\,) \eqd c^{H} X$ for all $c>0$. 
Stationary and self-similar random sup-measures are the only possible 
scaling limits of extremal processes on $[0,\infty)$, cf.~\cite{obr:tor:ver90}.

It is immediate that a CRSM is stationary (resp.\
self-similar) if and only if its extremal coefficient functional
satisfies $\theta(K+x)=\theta(K)$ (resp.\ $\theta(cK)=c^H \theta(K)$)
for $K\in \sK$.  However, the non-uniqueness of the spectral measure
$\pi$ in the LePage representation of max-stable random sup-measures
\eqref{eq:lepage-sm} 
implies that non-stationary (or non-selfsimilar) $Y$ may result in
stationary (or self-similar) random sup-measures. In particularly, the CRSM 
given by \eqref{eq:lepage-sm} is stationary if and only if $\E Y(K)=\E Y(K+x)$ 
for all $x\in\R^d$. In other words, the first order
stationarity of $Y$ implies the stationarity of $X$.  

\begin{example}
  Let $\carrier = \R$, and let $\zeta$ be a positive random variable
  with density $(rg(r))^{-1}$, $r>0$, for an appropriately chosen
  function $g$. Set $\Xi=\{\log \zeta\}$ and $\tau=g(\zeta)$ in
  Remark~\ref{rem:lepage-tm-non-integrable}, so that $Y(K)=g(\zeta)
  \ind_{\log \zeta \in K}$, $K\in\sK$. While $\Xi$ is not stationary,
  $\E Y(K)= \int_{0}^\infty \ind_{\log r \in K} \, r^{-1} dr$ equals
  the Lebesgue measure of $K$ and so is translation invariant.
\end{example}

\begin{remark}
  \label{remark:making-stationary}
  A generic construction of stationary CRSM works as follows. 
  Considering the Poisson process
  $\{(t_i,v_i,F_i)\}$ on $\R_+\times\R^d\times\sF'$ with the intensity
  measure $\lambda\otimes\lambda_d\otimes\nu$ (where $\lambda_d$ is
  the Lebesgue measure in $\R^d$) and let
  \begin{equation}
    \label{eq:stat-crsm}
    X(K)=\bigvee_{i\geq1} t_i^{-1}\ind_{(F_i+v_i)\cap
      K\neq\emptyset},\qquad K\in\sK. 
  \end{equation}
  The extremal coefficient functional of $X$ is 
  \begin{align*}
    \theta(K)=\int_{\sF'}\lambda_d(K+\check{F})\nu(dF),
  \end{align*}
  where $\check F=\{-x :\; x \in F\}$.
  If $\theta$ is normalised and corresponds to the random closed set
  $\Xi$, then the latter simplifies to
  $\theta(K)=\E\lambda_d(K+\check{\Xi})$.

  A similar construction with $F_i+v_i$ replaced by $s_iF_i$ on
  $\carrier=\R^d\setminus\{0\}$ for the Poisson process $\{s_i,
  i\geq1\}$ of intensity $\alpha s^{\alpha-1}ds$, $s>0$, yields
  self-similar CRSMs. 
  These constructions can be also applied to obtain stationary and
  self-similar versions of the tail dependence functional of general
  max-stable random sup-measures.
\end{remark}

\vspace{10mm}
\section{Examples of CRSM sup-measures}
\label{sec:examples-tm-sup}

\begin{example}[TM random vectors]
  \label{ex:tm-vectors}
  If $\carrier=\{1,\dots,d\}$ is a finite set, then the CRSM
  corresponds to a semi-simple max-stable random vector,
  whose distribution is uniquely determined by its extremal
  coefficients $\theta(K)=\ell(\ind_K)$, $K \subset \{1,\dots,d\}$.
  The comonotonicity property of $\ell$ is equivalent to
  \begin{align*}
    \ell(u)=(u_d-u_{d-1})\theta(\{d\})+(u_{d-1}-u_{d-2})\theta(\{d-1,d\})
    + \cdots +u_1\theta(\{1,\dots,d\})
  \end{align*}
  for $u_1\leq\cdots\leq u_d$.  Thus, the CRSMs on a finite
  carrier space become TM random vectors studied in
  \cite{str:sch15}. In this case, each CRSM is necessarily
  finite and the series representation \eqref{eq:lepage-tm-finite}
  yields 
  the series representation of TM random vectors from
  \cite{str:sch15}. Proposition~\ref{prob:completely-random} 
  becomes \cite[Eq.~(10)]{str:sch15}. It is easy to see that a stable
  sup-measure $X$ is a CRSM if and only if its finite dimensional
  distributions are TM random vectors.
\end{example}

\begin{example}
  If $\Xi=F$ is a deterministic closed set, then
  $\theta(K)=\ind_{K\cap F\neq\emptyset}$, and the corresponding CRSM
  constructed by \eqref{eq:lepage-tm-finite} is the indicator
  sup-measure $X(K)=\zeta \ind_{K\cap F\neq\emptyset}$, where $\zeta$
  is the unit Fr\'echet random variable.
\end{example}

\begin{example}
  Let $\Xi=[\zeta,\infty)$ on $\carrier=\R_+$, where $\zeta$ is a
  non-negative random variable, so that $\theta(K)=\Prob{\zeta\leq
    \sup K}$ is the capacity functional of $\Xi$. 
   The corresponding
  CRSM is given by $X(K)=\eta(\sup K)$ for the increasing max-stable
  process
  \begin{align*}
    \eta(t)=\bigvee_{i\geq1} \Gamma_i^{-1} \ind_{\zeta_i\leq t},\qquad t\geq0,
  \end{align*}
  where $\{\zeta_i\}$ are i.i.d.\ copies of $\zeta$.
\end{example}

\begin{example}
  Assume that $\carrier=\R^d$ and let $\Xi=\xi+M$, where $\xi$ is a random
  vector and $M$ is a deterministic compact set. The CRSM
  constructed by \eqref{eq:lepage-tm-finite} using i.i.d.\ copies of
  $\Xi$ can be obtained as $X(K)=Z(K+\check{M})$, where $Z$ is a
  completely random sup-measure with the control measure being the
  distribution of $\xi$ and $K+\check{M}=\{x - y : x \in K, y \in
  M\}$. The extremal coefficient functional of $X$ is
  $\theta(K)=\Prob{\xi\in K+\check{M}}$. 
\end{example}

\begin{example}
  \label{ex:perimeter}
  Let $\theta(K)$ be the perimeter of a convex set $K$ in
  $\carrier=\R^2$. The corresponding measure $\nu$ on $\sF'$ such that
  $\nu(\sF_K)=\theta(K)$ is the Haar measure on the affine
  Grassmannian $A(1,2)$ that consists of all lines in $\R^2$, see
  e.g.\ \cite[p.~582]{sch:weil08}. The measure $\lambda\otimes\nu$
  defines a stationary and isotropic marked line process
  $\{(t_i,L_i)\}$ on $\R_+\times A(1,2)$, see
  \cite[p.~124]{sch:weil08}. The LePage series
  \eqref{eq:rep-tm-sup-pp} defines a CRSM $X$ such that $X(K)$ equals
  the maximum of $t_i^{-1}$ for all lines $L_i$ that hit $K$.
\end{example}

\begin{example}
  \label{ex:lacaux}
  Let $\Xi$ be a random closed subset of $\carrier=\R_+$ with
  distribution $\mu$. For $\beta\in (0,1)$, define
  \begin{align*}
    \theta(K)=\int_0^\infty \Prob{\Xi+v \cap K\neq\emptyset}
    \beta v^{\beta-1}d v\,.
  \end{align*}
  The LePage representation of the corresponding CRSM $X$
  turns into 
  \begin{align*}
    X(K)\eqd \bigvee_{i\geq1}t_i^{-1} \ind_{\Xi_i+v_i\cap
      K\neq\emptyset},\qquad K\in\sK,
  \end{align*}
  where $\{(t_i,v_i,\Xi_i)\}$ is the Poisson process on
  $\R_+\times\R_+\times\sF'$ with the intensity $dt\beta
  v^{\beta-1}dvd\mu$.  This sup-measure $X$ is the central object in
  \cite{lac:sam15} for $\Xi$ being the range of a stable subordinator
  of order $(1-\beta)$. Then $\Xi$ coincides in distribution with
  $s\Xi$ for all $s>0$, so that $\theta(sK)=s^\beta \theta(K)$ for all
  $K\in\sK$ and $s>0$. Furthermore, $\theta(K+s)=\theta(K)$ meaning
  that $X$ is stationary. This CRSM $X$ is the sup-vague
  limit of appropriately rescaled random sup-measures arising from a
  stationary symmetric $\alpha$-stable sequence (here $\alpha=1$)
  whose dynamics is
  driven by a Markov chain with regularly varying first entrance time.
\end{example}

\begin{example}
  Let $\theta$ be the capacity functional of the random set $\Xi$
  being the path of the standard Brownian motion $W_t$, $t\geq0$, in
  $\R^d$ for $d\geq 3$ that starts at zero. The corresponding CRSM is
  constructed by \eqref{eq:lepage-tm-finite} and has the tail
  dependence functional $\ell(f)=\E\sup_{t\geq 0} f(W_t)$ for
  $f\in\USC$. 
\end{example}

\begin{example}
  \label{ex:C-additivity}
  Assume that $\carrier=\R^d$.  The measure $\nu$ on $\sF'$ related to
  the extremal coefficient functional by \eqref{eq:nu-theta-1} is
  supported by convex sets if and only if $\theta$ is additive on
  convex sets meaning that
  \begin{align*}
    \theta(K_1\cup K_2)+\theta(K_1\cap K_2)
    =\theta(K_1)+\theta(K_2)
  \end{align*}
  for all convex $K_1,K_2$ such that $K_1\cup K_2$ is also convex, see
  \cite[Th.~5.1.1]{mat75}. This property is also known under the name
  of \emph{C-additivity}, such functional $\theta$ is also called a
  \emph{valuation}, see \cite[Ch.~6]{schn2}. Assuming that $\theta$ is
  monotone and invariant for rigid motions (equivalently $X$ is
  stationary and isotropic), the Hadwiger theorem
  \cite[Th.~4.2.7]{schn2} establishes that 
  \begin{align}
    \label{eq:hadwiger}
    \theta(K)=\sum_{i=0}^d a_i V_i(K)
  \end{align}
  for all convex compact $K$, where $a_1,\dots,a_d\geq0$ and
  $V_0(K),\dots,V_d(K)$ are intrinsic volumes of $K$.

  The functional $a_i V_i(K)$ defines a Poisson process of intensity
  $a_i H_{d-i}$, where $H_{d-i}$ is the normalised Haar measure
  on the affine Grassmannian $A(d-i,d)$ that consists of all
  $(d-i)$-dimensional affine subspaces of $\R^d$. Thus, $\theta$
  yields a measure $\nu$ on $\sF'$ that corresponds to a superposition
  of such processes on Grassmannians of varying dimension, see also
  \cite[Th.~5.4.2]{mat75}. The sets $F_i$ in \eqref{eq:rep-tm-sup-pp}
  are affine subspaces of $\R^d$.
\end{example}

\begin{example}
  Let $\theta(K)=c\E \lambda_d(\Xi+\check{K})$ for a constant $c>0$
  and a random compact set $\Xi$ with distribution $\nu$. The
  corresponding measure $\nu$ on $\sF'$ is translation invariant and
  so can be disintegrated into the product $\lambda_d\otimes
  c\nu$. Then the LePage series from \eqref{eq:rep-tm-sup-pp} can be
  obtained from the Poisson process $\{(t_i,x_i,F_i)\}$ in
  $\R_+\times\R^d\times\sK$ with the intensity measure
  $\lambda\otimes\lambda_d\otimes\nu$, so that \eqref{eq:stat-crsm}
  yields
  a stationary CRSM with the extremal
  coefficient functional $\theta$.  In order to ensure that $X$ is
  finite on compact sets, it is required that the Lebesgue measure of
  the sum of $\Xi$ and the unit ball is integrable.  The
  sup-derivative $\xi$ of $X$ is the so-called \emph{storm process}
  generated by indicator functions, see \cite{lan:bac:bel11}. The random
  set $\Xi$ determines the shape of the random indicator function
  called a storm, while the points $x_i$ control the locations of
  storms whose strengths are then given by $t_i^{-1}$. In
  \cite{lan:bac:bel11} the random set $\Xi$ is chosen to be the Poisson
  polygon.

  The functional $\theta$ is additive on convex sets, and so admits
  the representation \eqref{eq:hadwiger}. For example, if
  $\Xi=B_\xi$ is the ball of random radius $\xi$ centred at the
  origin, then the Steiner formula from convex geometry
  \cite[p.~208]{schn2} yields that
  \begin{align*}
    \theta(K)=\sum_{i=0}^d V_i(K)\E\xi^{d-i}
  \end{align*}
  for each convex compact set $K$.  In particular, $X$ shares the same
  distribution with the CRSM from Example~\ref{ex:C-additivity} (with
  $a_i=\E\xi^{d-i}$) on any chain $K_1\subset K_2\subset \cdots\subset
  K_m$ of convex sets.
\end{example}

\begin{example} 
  Let $W$ be a centred Gaussian process on $\R^d$ with stationary
  increments, that is, the law of $\{W(x+y) - W(y)\}_{x \in \R^d}$
  does not depend on $y \in \R^d$. Specifying $W(0)=0$, the law of $W$ is
  uniquely determined by its variogram $\gamma(x,y)=\E (W(x)-W(y))^2$.
  The Brown--Resnick process associated to $\gamma$ \cite{kab:sch:dH09}
  is defined by the LePage series representation
  \begin{align*}
    \xi(x)=\bigvee_{i\geq 1} \Gamma_i^{-1} \exp\bigg(W_i(x)-\frac{\gamma(x,0)}{2}\bigg),
    \qquad x \in \R^d,
  \end{align*}
  where $\{W_i,i \geq 1\}$ are i.i.d.\ copies of $W$ independent
  of $\{\Gamma_i,i \geq 1\}$.  Its sup-integral $X(K)=\sup\{\xi(x)
  \,:\, x \in K\}$, $K \in \sK$, is a stationary
  max-stable random sup-measure.  However, it is not a CRSM, since its
  sup-derivative $\xi$ has continuous paths.  This also follows from
  the fact that $\xi(x_1),\dots,\xi(x_m)$ follows multivariate
  H\"usler-Reiss distributions \cite{hus:rei89} that are not
  spectrally discrete and so do not correspond to TM random vectors, cf.\
  Example~\ref{ex:tm-vectors}. This has been illustrated in
  \cite[Fig.~3]{str:sch15} by plotting the dependency sets (which are
  a finite-dimensional analogue of the sets of measures $\bM$ in
  \eqref{eq:ell-dual-measures}) and which are not polytopes.
\end{example}

{
\vspace{5mm}
\section*{Acknowledgements}
\label{sec:acknowledgements}
The authors are grateful to Zakhar Kabluchko, Martin Schla\-ther and Yizao Wang for discussions concerning this work and to the anonymous referees and the AE for critical remarks that led to numerous improvements.
}

\vspace{5mm}
\bibliographystyle{plainnat} 

{

}

\end{document}